\documentclass[11pt, a4paper]{amsart}
\usepackage{fullpage}
\usepackage{amssymb}
\usepackage{commath}
\usepackage{empheq}
\usepackage{dsfont}
\usepackage{amsmath, amsthm}
 \usepackage[foot]{amsaddr}
\usepackage[parfill]{parskip}

\usepackage{tikz-cd}
\usepackage{url}

\newtheorem{thm}{Theorem}[section]

\theoremstyle{plain}
\newtheorem{prop}[thm]{Proposition}

\newtheorem{cor}[thm]{Corollary}

\theoremstyle{definition}
\newtheorem{defn}[thm]{Definition}

\theoremstyle{remark}
\newtheorem{rem}[thm]{Remark}
\newtheorem{eg}[thm]{Example}

\newcommand{\RNum}[1]{\MakeUppercase{\romannumeral #1}}

\newcommand{\ol}{\overline}
\renewcommand{\phi}{\varphi}
\newcommand{\llb}{\left\lbrace}
\newcommand{\rrb}{\right\rbrace}

\newcommand{\llv}{\left\lvert}
\newcommand{\rrv}{\right\rvert}
\newcommand{\llangle}{\left\langle}
\newcommand{\rrangle}{\right\rangle}

\title{Reflexive homology}
\author{Daniel Graves}
\address{School of Mathematics, University of Leeds, Woodhouse, Leeds, LS2 9JT, UK}
\email{dan.graves92@gmail.com}
\date{}

\begin{document}


\maketitle

\begin{abstract}
Reflexive homology is the homology theory associated to the reflexive crossed simplicial group; one of the fundamental crossed simplicial groups. It is the most general way to extend Hochschild homology to detect an order-reversing involution. In this paper we study the relationship between reflexive homology and the $C_2$-equivariant homology of free loop spaces. We define reflexive homology in terms of functor homology. We give a bicomplex for computing reflexive homology together with some calculations, including the reflexive homology of a tensor algebra. We prove that the reflexive homology of a group algebra is isomorphic to the homology of the $C_2$-equivariant Borel construction on the free loop space of the classifying space. We give a direct sum decomposition of the reflexive homology of a group algebra indexed by conjugacy classes of group elements, where the summands are defined in terms of a reflexive analogue of group homology. We define a hyperhomology version of reflexive homology and use it to study the $C_2$-equivariant homology of certain free loop and free loop-suspension spaces. We show that reflexive homology satisfies Morita invariance. We prove that under nice conditions the involutive Hochschild homology studied by Braun and by Fern\`andez-Val\`encia and Giansiracusa coincides with reflexive homology. 
\end{abstract}

\section*{Introduction}
Reflexive homology is the homology theory associated to the reflexive crossed simplicial group. It is the most general way to extend Hochschild homology to detect an order-reversing involution.

The study of involutive structures in algebraic topology has been very fruitful in recent years with the development of \emph{real topological Hochschild homology} (\cite{DMP},\cite{DMPR}, \cite{Hogenhaven1}, \cite{Hogenhaven2}, \cite{AKGH}, \cite{DMP2}, \cite{dotto-thesis}), \emph{real algebraic $K$-theory} (\cite{HM}, \cite{DO}, \cite{Dotto-equivariant}) and a renaissance in the study of \emph{Hermitian $K$-theory} (\cite{HKT1}, \cite{HKT2}, \cite{HKT3}).

In this paper we study the relationship between reflexive homology and the $C_2$-equivariant homology of free loop spaces. The study of free loop spaces occurs widely in topology and geometry. In topology they play an important role in string topology (\cite{CS}) and topological Hochschild and cyclic homology (\cite{BCS}, \cite{NS}). See also \cite{CHV} for more on both of these topics. In geometry, the free loop space is intimately connected with the study of closed geodesics on manifolds (\cite{LF}, \cite{GM}). See also \cite{Oancea} for  a survey of results in this area.

Crossed simplicial groups were introduced independently by Fiedorowicz and Loday \cite{FL} and Krasauskas \cite{Kras-skew} in order to study equivariant homology. One way of thinking about a crossed simplicial group is as the structure required to build group actions into Hochschild homology in the same fashion as the cyclic homology theory due to Connes \cite{Con} (see also \cite{Lod}). Since their introduction, crossed simplicial groups have been well-studied and have found applications in other areas such as the categorification of monoids in symmetric and braided monoidal categories (\cite{Pir-PROP}, \cite{Lack}, \cite{DMG-IFAS}, \cite{DMG-PROB}, \cite{DMG-equiv}) and combinatorial models for marked surfaces with a $G$-structure (\cite{DK}).

A classification result (\cite[3.6]{FL}, \cite[1.5]{Kras-skew}) tells us that any crossed simplicial group occurs as an extension of a \emph{fundamental} crossed simplicial group. These fundamental crossed simplicial groups are subobjects of the hyperoctahedral crossed simplicial group \cite[Section 3]{FL}:  trivial; reflexive; cyclic; dihedral; symmetric; and hyperoctahedral. Most of these crossed simplicial groups have been well-studied. Whilst the associated homology theories have a range of interesting applications, for the purposes of this introduction we will restrict ourselves to results about loop spaces.

The homology theories associated to the trivial crossed simplicial group and the cyclic crossed simplicial group are Hochschild homology and cyclic homology respectively. These homology theories, when applied to the group algebra of a discrete group, calculate the homology and the $S^1$-equivariant homology of the free loop space on the classifying space of the group respectively \cite[7.3.13]{Lod}. Indeed, the connections between the cyclic homology theory and free loop spaces is well-established (\cite{DB-cyc}, \cite{Goodwillie}, \cite{Burg-Fie}, \cite{Jones}, \cite{VP-B}, \cite{CC}). The dihedral homology theory (\cite{Lod-dihed}, \cite{KLS-dihed}) is used to calculate the $O(2)$-equivariant homology of free loop spaces (\cite{Lodder1}, \cite{Dunn}, \cite{Ung}).

In spite of the fact that crossed simplicial groups and their associated homology theories are well-studied, the reflexive crossed simplicial group appears only ephemerally in the literature; usually only being considered insofar as it relates to the dihedral crossed simplicial group. For example, Krasauskas, Lapin and Solov'ev \cite[Section 3]{KLS-dihed} give a definition of reflexive homology in terms of hyperhomology in order to obtain a dihedral version of Connes' periodicity long exact sequence. Spali\'nski \cite[Section 3]{spal-discrete} shows that the category of reflexive sets admits a model structure that is Quillen equivalent to the category of $C_2$-spaces with the fixed-point model structure and uses this as a tool for giving a discrete model of $O(2)$-equivariant homotopy theory, which also arises from the dihedral crossed simplicial group.

In this paper we study the reflexive homology theory in its own right. The term ``reflexive" for this crossed simplicial group was first used in \cite[Proposition 1.5]{Kras-skew}. As the homology theory associated to a crossed simplicial group, reflexive homology is defined as functor homology over a small category which we will denote $\Delta R^{op}$. The structure of the indexing category $\Delta R^{op}$ can be thought of as encoding an order-reversing involution compatible with a unital, associative multiplication. In this sense, the reflexive homology theory offers the most general framework for extending Hochschild homology to detect the action of an involution.

The paper is structured as follows.

In Section \ref{ref-hom-sec} we define reflexive homology in terms of functor homology and use this to define the reflexive homology of an involutive algebra over a commutative ring. We provide the necessary background material on functor homology and Hochschild homology and survey the results that currently exist in the literature.

In Section \ref{bires-sec} we define a bicomplex that can be used to calculate reflexive homology. We use this to show that our functor homology definition coincides with the hyperhomology definition of \cite{KLS-dihed}. We also show that when working over a field of characteristic zero, reflexive homology can be calculated using the quotient of the Hochschild complex by the involution action.

In Section \ref{ref-hyperhom-sec} we define reflexive hyperhomology for chain complexes of left $\Delta R^{op}$-modules in terms of hyper-derived functors and recall some important examples.

In Section \ref{morita-sec} we prove that reflexive homology satisfies Morita invariance. Explicitly, we show that for an involutive algebra $A$, the reflexive homology of the involutive algebra of $(m\times m)$-matrices with entries in $A$ is isomorphic to the reflexive homology of $A$.

In Section \ref{calc-sec} we provide some computations. We calculate the reflexive homology of the ground ring and describe degree zero reflexive homology for commutative algebras with involution.

In Section \ref{tensor-alg-sec} we calculate the reflexive homology of a tensor algebra. We show that this can be described in terms of the group homology of $C_2$ with coefficients in the Hochschild homology of the tensor algebra, as calculated by Loday and Quillen \cite{LQ}. As a consequence, we show that the reflexive homology of a tensor algebra has a grading.

In Section \ref{kg-sec} we calculate the reflexive homology of a group algebra. We prove that this is isomorphic to the homology of the $C_2$-equivariant Borel construction on the free loop space of the classifying space of the group. By combining this with the cyclic homology of a group algebra \cite[7.3.13]{Lod}, we show that the dihedral homology of a group algebra is isomorphic to the homology of the $O(2)$-equivariant Borel construction on the free loop space of the classifying space of the group. Furthermore, we give a direct sum decomposition of the reflexive homology of a group algebra, indexed over the conjugacy classes of the group, where the summands are given in terms of a reflexive analogue of group homology.

In Section \ref{Moore-loop-sec} we use our reflexive hyperhomology to prove that we can calculate the $C_2$-equivariant homology of certain free loop spaces and free loop-suspension spaces in terms of the singular chain complex on certain Moore loop spaces.

Other constructions that build an involution into Hochschild homology exist in the literature. Braun \cite{Braun} introduced \emph{involutive Hochschild homology} in order to study involutive algebras and involutive $A_{\infty}$-algebras. The homological algebra of this theory was developed by Fern\`andez-Val\`encia and Giansiracusa \cite{RFG}. In particular, they show that under nice conditions involutive Hochschild homology can be described as Tor over an involutive analogue of the enveloping algebra. In Section \ref{ihh-sec} we show that under these conditions there is an isomorphism between involutive Hochschild homology and reflexive homology.

In Section \ref{THR-section}, we provide some exposition on how the structure of the reflexive crossed simplicial group appears in the study of \emph{real topological Hochschild homology}. In particular, we observe that \emph{real simplicial objects} in a category are precisely the same as reflexive objects in a category. An important example of this is the dihedral nerve construction, which plays an important role in defining real topological Hochschild homology.

\section*{Acknowledgements}
I would like to thank James Brotherston, James Cranch, Callum Reader and Sarah Whitehouse for interesting and helpful conversations at various stages of writing this paper. I would like to thank the referee for their helpful comments and suggestions.

\section*{Conventions}
Throughout the paper we will let $k$ be a commutative ring. An unadorned tensor product symbol, $\otimes$, will denote the tensor product of $k$-modules. The category of $k$-modules will be denoted by $\mathbf{Mod}_k$. We will denote by $\mathbf{Top}$ the category of compactly-generated weak Hausdorff topological spaces. Let $\mathbf{Top}_{\star}$ denote the category of based compactly-generated weak Hausdorff topological spaces. We will usually refer to a ``topological space" or a ``based topological space". When referring to a weak equivalence of topological spaces we mean a $\pi_{\star}$-isomorphism. Several of our results relate to free loop spaces so we will introduce notation for this here. Let $\mathcal{L}$ denote the functor $\mathrm{Maps}\left(S^1, -\right)$, which sends a based topological space $X$ to the space of unbased continuous maps $S^1\rightarrow X$.

\section{Reflexive homology}
\label{ref-hom-sec}

In this section we recall the definition of the reflexive crossed simplicial group and define its associated homology theory. 

\subsection{Functor homology and Hochschild homology}
We start by recalling some constructions from functor homology, using Hochschild homology as an example.

For a small category $\mathbf{C}$ there are abelian categories $\mathbf{CMod}=\mathrm{Fun}\left(\mathbf{C}, \mathbf{Mod}_k\right)$ and $\mathbf{ModC}=\mathrm{Fun}\left(\mathbf{C}^{op}, \mathbf{Mod}_k\right)$. There is a tensor product
\[-\otimes_{\mathbf{C}} - \colon \mathbf{ModC} \times \mathbf{CMod} \rightarrow \mathbf{Mod}_k\]
defined as the coend
\[G\otimes_{\mathbf{C}} F = \int^{C\in \mathrm{Ob}(\mathbf{C})} G(C) \otimes_k F(C)\]
as in \cite[Section 3]{MCM}. It is well-known that this tensor product is right exact with respect to both variables and preserves direct sums \cite[Section 1.6]{Pir02}. The left derived functors of this tensor product are denoted by $\mathrm{Tor}_{\star}^{\mathbf{C}}\left(-,-\right)$. When $G=k^{\ast}$, the constant functor at $k$, we write $H_{\star}\left(\mathbf{C},F\right)= \mathrm{Tor}_{\star}^{\mathbf{C}}\left(k^{\ast}, F\right)$.

As an example, we can recover Hochschild homology of a simplicial $k$-module. Recall the category $\Delta$, whose objects are the sets $[n]=\llb 0,\dotsc , n\rrb$ for $n\geqslant 0$ and whose morphisms are order-preserving maps \cite[B.1]{Lod}. If we take $\mathbf{C}=\Delta^{op}$ and a simplicial $k$-module $F\colon \Delta^{op}\rightarrow \mathbf{Mod}_k$ we recover Hochschild homology. One example that we will be particularly interested in is studying $k$-algebras so we recall the \emph{Loday functor}.

Let $A$ be an associative $k$-algebra and let $M$ be an $A$-bimodule. There is a functor 
\[\mathcal{L}(A,M)\colon \Delta^{op}\rightarrow \mathbf{Mod}_k\]
given on objects by $[n]\mapsto M\otimes A^{\otimes n}$ and determined on morphisms by 
\[\partial_i\left(m\otimes a_1\otimes \cdots \otimes a_n\right) =
\begin{cases}
\left( ma_1\otimes a_2\otimes \cdots \otimes a_n\right) & i=0\\
\left(m\otimes a_1\otimes \cdots \otimes a_ia_{i+1}\otimes \cdots \otimes a_n\right) & 1\leqslant i \leqslant n-1\\
\left(a_nm\otimes a_1\otimes \cdots \otimes a_{n-1}\right) & i=n
\end{cases}
\]
and
\[s_j\left(m\otimes a_1\otimes \cdots \otimes a_n\right) = 
\begin{cases}
\left(m\otimes 1_A \otimes a_1\otimes \cdots \otimes a_n\right) &j=0\\
\left(m\otimes a_1\otimes \cdots\otimes a_j \otimes 1_A \otimes a_{j+1} \otimes \cdots \otimes a_n\right) & j\geqslant 1.
\end{cases}
\]

The chain complex associated to this simplicial $k$-module is the Hochschild complex $C_{\star}(A,M)$ and its homology is Hochschild homology, $HH_{\star}(A,M)$. In particular we have 
\[HH_{\star}(A,M)= H_{\star}\left(\Delta^{op} , \mathcal{L}(A,M)\right) = \mathrm{Tor}_{\star}^{\Delta^{op}}\left(k^{\ast}, \mathcal{L}\left(A,M\right)\right).\]

\subsection{Reflexive homology}
An important source of functor homology theories come from \emph{crossed simplicial groups}, introduced independently by Fiedorowicz and Loday \cite{FL} and Krasauskas \cite{Kras-skew}.

In this paper we will study the homology theory associated to the reflexive crossed simplicial group. 
\begin{defn}
Let $R_n=\llangle r_n\mid r_n^2=1\rrangle \cong C_2$ for $n\geqslant 0$.
\end{defn}

The family of groups $\llb R_{\star}\rrb$ forms a crossed simplicial group \cite[Example 2]{FL}, whose geometric realization is $C_2$.

\begin{defn}
\label{delta-r-defn}
The category $\Delta R$ has the sets $[n]=\llb 0,\dotsc , n\rrb$ for $n\geqslant 0$ as objects. An element of $\mathrm{Hom}_{\Delta R}\left([n],[m]\right)$ is a pair $\left(\phi , g\right)$ where $g\in R_n$ and $\phi \in \mathrm{Hom}_{\Delta}\left([n],[m]\right)$. The composition is determined as follows. For a face map $\delta_i\in \mathrm{Hom}_{\Delta}\left([n],[n+1]\right)$ we define $r_{n+1}\circ\delta_i=\delta_{n-i}\circ r_n$ and for a degeneracy map $\sigma_j\in \mathrm{Hom}_{\Delta}\left([n],[n-1]\right)$ we define $r_{n}\circ \sigma_j=\sigma_{n-j}\circ r_{n-1}$. 
\end{defn}

\begin{rem}
The category $\Delta R$ also appears in \cite[Section 1]{KLS-dihed} with the notation $\Delta \ltimes \mathbb{Z}/2$.
\end{rem}

\begin{defn}
\label{refhom-defn}
Let $F\in \mathrm{Fun}\left(\Delta R^{op},\mathbf{Mod}_k\right)$. We define the \emph{reflexive homology} of $F$ to be $HR_{\star}\left(F\right)=\mathrm{Tor}_{\star}^{\Delta R^{op}}\left(k^{\ast}, F\right)$, where $k^{\ast}\in \mathrm{Fun}\left(\Delta R, \mathbf{Mod}_k\right)$ is the constant functor at $k$.
\end{defn}

We can extend the Loday functor to a functor $\mathcal{L}(A,M)\colon \Delta R^{op}\rightarrow \mathbf{Mod}_k$ in two different ways but first we require some definitions.

\begin{defn}
A $k$-algebra is said to be \emph{involutive} if it is equipped with an anti-homomorphism of algebras $A\rightarrow A$ of order two, which we will denote by $a\mapsto \ol{a}$. In other words, we equip $A$ with a $k$-linear endomorphism which reverses the order of multiplication and squares to the identity.
\end{defn}

Following Loday \cite[5.2.1]{Lod} we recall the notion of an involutive $A$-bimodule.

\begin{defn}
Let $A$ be an involutive $k$-algebra. An \emph{involutive $A$-bimodule} is an $A$-bimodule $M$ equipped with a map $m\mapsto \ol{m}$ such that $\ol{a_1ma_2}= \ol{a_2}\, \ol{m}\, \ol{a_1}$ for $a_1$, $a_2\in A$.
\end{defn}

\begin{eg}
Taking $M=A$ gives one example of an involutive $A$-bimodule. If $A$ is equipped with an augmentation $\varepsilon\colon A\rightarrow k$ we can also take $M=k$ with the trivial involution.
\end{eg}

With these definitions we can extend the Loday functor to $\Delta R^{op}$ in two ways.

\begin{defn}
\label{loday-defn}
Let $A$ be an involutive $k$-algebra and let $M$ be an involutive $A$-bimodule. We extend the Loday functor $\mathcal{L}(A,M)$ to a functor 
\[\mathcal{L}^{+}\left(A,M\right)\colon \Delta R^{op}\rightarrow \mathbf{Mod}_k\] by defining
\[r_n\left(m\otimes a_1\otimes \cdots \otimes a_n\right) = \left(\ol{m}\otimes \ol{a_n}  \otimes \cdots \otimes \ol{a_1}\right).\]
\end{defn}

\begin{defn}
Let $A$ be an involutive $k$-algebra and let $M$ be an involutive $A$-bimodule. We define
\[HR_{\star}^{+}(A,M)\coloneqq \mathrm{Tor}_{\star}^{\Delta R^{op}}\left(k^{\ast},\mathcal{L}^{+}(A,M)\right).\]
\end{defn}

\begin{defn}
Let $A$ be an involutive $k$-algebra and let $M$ be an involutive $A$-bimodule. We extend the Loday functor $\mathcal{L}(A,M)$ to a functor 
\[\mathcal{L}^{-}\left(A,M\right)\colon \Delta R^{op}\rightarrow \mathbf{Mod}_k\] by defining
\[r_n\left(m\otimes a_1\otimes \cdots \otimes a_n\right) = -\left(\ol{m}\otimes \ol{a_n}  \otimes \cdots \otimes \ol{a_1}\right).\]
\end{defn}

\begin{defn}
Let $A$ be an involutive $k$-algebra and let $M$ be an involutive $A$-bimodule. We define
\[HR_{\star}^{-}(A,M)\coloneqq \mathrm{Tor}_{\star}^{\Delta R^{op}}\left(k^{\ast},\mathcal{L}^{-}(A,M)\right).\]
\end{defn}

\begin{rem}
When we take $M=A$ we will omit the coefficients from the notation and write $\mathcal{L}^{\pm}(A)$ and $HR_{\star}^{\pm}(A)$.
\end{rem}

\begin{rem}
\label{LES-rem}
These reflexive homology groups fit into long exact sequences with the dihedral homology groups by \cite[Proposition 3.1]{KLS-dihed}. There exist long exact sequences
\[\cdots \rightarrow HR_n^{+}(A)\rightarrow HD_n^{+}(A)\rightarrow HD_{n-2}^{-}(A)\rightarrow HR_{n-1}^{+}(A)\rightarrow \cdots\]
and
\[\cdots \rightarrow HR_n^{-}(A)\rightarrow HD_n^{-}(A)\rightarrow HD_{n-2}^{+}(A)\rightarrow HR_{n-1}^{-}(A)\rightarrow \cdots\]
where the dihedral homology groups $HD_{\star}^{+}(A)$ and $HD_{\star}^{-}(A)$ are defined in \cite[Section 1]{KLS-dihed}.
Furthermore, as discussed in \cite[5.2]{Lod}, if the ground ring $k$ contains $\mathbb{Q}$ then the direct sum of these long exact sequences is Connes' long exact sequence connecting Hochschild homology and cyclic homology \cite[Theorem 2.2.1]{Lod}.
\end{rem}

\begin{rem}
For an involutive $k$-algebra $A$, there is natural map $HR_{\star}^{+}(A) \rightarrow HO_{\star}(A)$, where $HO_{\star}$ denotes the hyperoctahedral homology theory of \cite{DMG-hyp} and \cite[Section 2]{Fie}. One can obtain this map by composing the map $HR_{\star}^{+}(A) \rightarrow HD_{\star}^{+}(A)$ from the previous remark with the map $HD_{\star}\left(A\right) \rightarrow HO_{\star}\left(A\right)$ of \cite[Subsection 3.3]{DMG-hyp}.
\end{rem}

\section{Biresolution}
\label{bires-sec}
We begin this section by defining a biresolution of the $k$-constant right $\Delta R^{op}$-module $k^{\ast}$. Fiedorowicz and Loday \cite[6.7]{FL} give a general construction for such a biresolution using the homogeneous bar resolution for a group \cite[\RNum{7}.6]{CWM}. The bicomplex defined here is smaller, making use of the periodic resolution of the group $C_2$. It is also worth noting that our bicomplex includes into the tricomplex of Krasauskas, Lapin and Solov'ev \cite[Section 2]{KLS-dihed} for computing dihedral homology.

\begin{defn}
\label{bires-defn}
We define a bicomplex $C_{\star,\star}$ of right $\Delta R^{op}$-modules as follows. Firstly we set
\[C_{p,q}=k\left[\mathrm{Hom}_{\Delta R}\left([q],-\right)\right]\]
for all $p,\,q\geqslant 0$.

The horizontal differential $d\colon C_{p,q}\rightarrow C_{p-1,q}$ for $q\geqslant 0$ and $p\geqslant 1$ is given by
\[d=\begin{cases}
1-r_q & q\equiv 0,\, 3\,(\mathrm{mod}\, 4), \, p\equiv 1\, (\mathrm{mod}\, 2)\\
1+r_q & q\equiv 1,\, 2\,(\mathrm{mod}\, 4), \, p\equiv 1\, (\mathrm{mod}\, 2)\\
1+r_q & q\equiv 0,\, 3\,(\mathrm{mod}\, 4), \, p\equiv 0\, (\mathrm{mod}\, 2)\\
1-r_q & q\equiv 1,\, 2\,(\mathrm{mod}\, 4), \, p\equiv 0\, (\mathrm{mod}\, 2)
\end{cases}
\]
where the map $r_q$ is defined by pre-composition.
  
The vertical differential $b\colon C_{p,q}\rightarrow C_{p,q-1}$ for $p\geqslant 0$ and $q\geqslant 1$ is given by
\[b=\sum_{i=0}^q (-1)^i \delta_i\]
where the maps $\delta_i$ are defined by pre-composition.
\end{defn}

\begin{prop}
\label{bires-prop}
The bicomplex $C_{\star,\star}$ of Definition \ref{bires-defn} is a biresolution of the $k$-constant right $\Delta R^{op}$-module $k^{\ast}$.
\end{prop}
\begin{proof}
It suffices to show that the bicomplex of $k$-modules obtained by evaluating $C_{\star,\star}$ on each object $[n]$ of $\Delta R^{op}$ is a resolution of $k$.

We fix an object $[n]$. We observe that $C_{p,q}([n])$ is a free $k[C_2]$-module on the set of generators $\mathrm{Hom}_{\Delta}\left([q],[n]\right)$ for each $p\geqslant 0$.

Consider the $E^1$-page of the spectral sequence obtained by taking the horizontal homology of the bicomplex $C_{\star,\star}([n])$. The rows are exact complexes of finitely-generated $k[C_2]$-modules. Therefore, the $E^1$-page is isomorphic to the complex 
\[k\left[\mathrm{Hom}_{\Delta}\left([\star],[n]\right)\right]\]
with differential $b$ concentrated in the column $p=0$. This is the complex associated to the standard simplicial model of the $n$-simplex, which is acyclic. We deduce that the $E^2$-page of the spectral sequence is isomorphic to a copy of $k$ concentrated in bidegree $(0,0)$ as required.
\end{proof} 

In the introduction we said that Krasauskas, Lapin and Solov'ev \cite[Section 3]{KLS-dihed} had given a definition of reflexive homology in terms of hyperhomology. We now demonstrate that the functor homology definition coincides with that definition.

\begin{prop}
\label{hyperhom-prop}
Let $F\in \mathrm{Fun}\left(\Delta R^{op},\mathrm{Mod}_k\right)$. The reflexive homology $HR_{\star}(F)$ is naturally isomorphic to the hyperhomology of the group $C_2$ with coefficients in the Hochschild complex $C_{\star}(F)$.
\end{prop}
\begin{proof}
We observe that for $F\in \mathrm{Fun}\left(\Delta R^{op},\mathrm{Mod}_k\right)$, the Hochschild complex $C_{\star}(F)$ is a complex of $k[C_2]$-modules. Consider the bicomplex $C_{\star,\star}\otimes_{\Delta R^{op}} F$. On the one hand, by definition, the homology of this bicomplex is $HR_{\star}(F)$. On the other hand, it is a bicomplex of $k$-modules with the Hochschild complex $C_{\star}(F)$ in the column $p=0$ such that the homology of row $n$ for $n\geqslant 0$ is the group homology of $C_2$ with coefficients in $F([n])$. In other words, it is a bicomplex of $k$-modules which computes the hyperhomology of $C_2$ with coefficients in the complex $C_{\star}(F)$.
\end{proof}

\begin{prop}
\label{char-zero-prop}
Suppose that $2$ is invertible in the ground ring. Let $A$ be an involutive $k$-algebra and let $M$ be an involutive $A$-bimodule. There exist isomorphisms of graded $k$-modules
\[HR_{\star}^{+}(A,M)\cong H_{\star}\left(\frac{C_{\star}(A,M)}{(1-r)} \right)\quad \text{and} \quad  HR_{\star}^{-}(A,M)\cong H_{\star}\left(\frac{C_{\star}(A,M)}{(1+r)}\right).\]
\end{prop}
\begin{proof}
Consider the horizontal homology spectral sequences of the bicomplexes of $k$-modules $C_{\star ,\star}\otimes_{\Delta R^{op}}\mathcal{L}^{+}(A,M)$ and $C_{\star, \star}\otimes_{\Delta R^{op}} \mathcal{L}^{-}(A,M)$. In each case, the homology of the rows of the bicomplex is $H_{\star}\left(C_2, M\otimes  A^{\otimes n}\right)$, for the given actions of $C_2$ on $M\otimes A^{\otimes n}$. Since $2$ is invertible in $k$ and $C_2$ is finite, $H_{n}\left(C_2, M \otimes A^{\otimes n}\right)=0$ for $n\geqslant 1$. Therefore, in each case, the $E^1$-page consists of the Hochschild complex with the $C_2$ action factored out, concentrated in the column $p=0$.
\end{proof}

\section{Reflexive hyperhomology}
\label{ref-hyperhom-sec}
We can extend the definition of reflexive homology to chain complexes of left $\Delta R^{op}$-modules by defining reflexive hyperhomology. This is a property common to all crossed simplicial groups, as described in \cite[Section 3]{Dunn} for example. We define reflexive hyperhomology in terms of hypertor functors (see \cite[5.7.8]{weib} for instance) over the category $\Delta R^{op}$. In our case, these can be described explicitly as the hyper-derived functors of the tensor product $k^{\ast} \otimes_{\Delta R^{op}}-$ where $k^{\ast}$ is the $k$-constant right $\Delta R^{op}$-module. 

\begin{defn}
A non-negatively graded reflexive chain complex is a functor $F\colon \Delta R^{op}\rightarrow \mathbf{ChCpx}$, where $\mathbf{ChCpx}$, is the category of non-negatively graded chain complexes of $k$-modules. Equivalently, a reflexive chain complex is a non-negatively graded chain complex of left $\Delta R^{op}$-modules.
\end{defn}

\begin{eg}
One example of a reflexive chain complex arises by composing a reflexive topological space $X\colon \Delta R^{op} \rightarrow \mathbf{Top}$ with the singular chain complex functor $S_{\star}(-,k)$.
\end{eg}

\begin{eg}
\label{sing-chain-eg2}
Another important example of a reflexive chain complex arises from involutive DGAs. Recall that a DGA (\emph{differential graded algebra}, or sometimes \emph{chain algebra}), $(A,d)$, is a graded $k$-algebra, $A$, equipped with a $k$-linear map $d\colon A \rightarrow A$ satisfying the following properties. The map $d$ has degree $-1$, it squares to the identity and it satisfies the graded Leibniz rule: $d(ab)=(da)b+(-1)^{\llv a \rrv}a(db)$, where $\llv a \rrv$ denotes the degree of a homogeneous element $a$.

An involutive DGA, $(A,d,i)$, is a DGA, $(A,d)$, equipped with a chain map 
\[i\colon (A,d)\rightarrow (A,d),\]
written $i(a)=\ol{a}$. The map $i$ must square to the identity and satisfy $\ol{ab}= (-1)^{\llv a \rrv\cdot\llv b \rrv} \ol{b}\ol{a}$.

Given an involutive DGA $(A,d,i)$, we can form its reflexive bar construction, $\Gamma\left(A,d,i\right)$, using the simplicial and reflexive structure described in \cite[Section 3]{Dunn}. The $n$-simplices are given by $(A,d)^{\otimes (n+1)}$. Let $a=a_0\otimes \cdots \otimes a_n$. The face maps are given by 
\[\partial_i(a)=\begin{cases}
a_0\otimes \cdots \otimes a_ia_{i+1}\otimes \cdots \otimes a_n & 0\leqslant i \leqslant n-1\\
(-1)^{\llv a_n\rrv\left(\llv a_0\rrv +\cdots +\llv a_{n-1} \rrv\right)} a_na_0\otimes a_1\otimes \cdots \otimes a_{n-1} & i=n.
\end{cases}
\] 
The degeneracy maps are given by 
\[s_j(a)= a_0\otimes \cdots \otimes a_j \otimes 1 \otimes a_{j+1} \otimes \cdots \otimes a_n\]
for $0\leqslant j\leqslant n$.

The reflexive operators are defined by 
\[r(a)= (-1)^{\lvert\lvert a \rvert\rvert\cdot n(n-1)/2}\, \ol{a_0}\otimes \ol{a_n}\otimes \cdots \otimes \ol{a_1}\]
where $\lvert\lvert a \rvert\rvert = \sum_{i=1}^{n-1} \sum_{j=i+1}^n \llv a_i\rrv \cdot \llv a_j\rrv$. 

This is another example of a reflexive chain complex.

One example that will be of particular interest in Section \ref{moore-loop-free-loop-subsec} is the following. If $M$ is an involutive topological monoid, the singular chain complex $S_{\star}(M,k)$ is an involutive DGA with the involution induced from the one on $M$. 
\end{eg}

We now define reflexive homology of a chain complex of left $\Delta R^{op}$-modules

\begin{defn}
\label{hyperhom-defn}
Let $k_0^{\ast}$ denote the right $\Delta R^{op}$-module $k^{\ast}$ thought of as a chain complex concentrated in degree zero. Let $F_{\star}$ be a non-negatively graded chain complex of left $\Delta R^{op}$-modules. For each $n\geqslant 0$, we define the $n^{th}$ reflexive homology of $F_{\star}$ by 
\[HR_{n}\left(F_{\star}\right) = \mathbf{Tor}_{n}^{\Delta R^{op}}\left(k_0^{\ast} , F_{\star}\right).\]
\end{defn}

\begin{rem}
We can use the biresolution of Definition \ref{bires-defn} to construct a chain complex for computing the reflexive homology of a reflexive chain complex. Since the chain complex $k_0^{\ast}$ is concentrated in degree zero, the tensor product of chain complexes $k_0^{\ast}\otimes_{\Delta R^{op}} F_{\star}$ is equivalent to applying the functor $k^{\ast}\otimes_{\Delta R^{op}}-$ to the chain complex $F_{\star}$ degree-wise. Replacing $k^{\ast}$ by the biresolution of Definition \ref{bires-defn} we obtain a chain complex of bicomplexes, that is, a tricomplex. Applying the total complex functor for tricomplexes we obtain a chain complex which calculates $\mathbf{Tor}_{n}^{\Delta R^{op}}\left(k_0^{\ast} , F_{\star}\right)$.
\end{rem}

\begin{rem}
We observe that if we consider a left $\Delta R^{op}$-module as a chain complex concentrated in degree zero then the reflexive hyperhomology given in Definition \ref{hyperhom-defn} coincides with the definition of reflexive homology given in Definition \ref{refhom-defn}.
\end{rem}

\section{Morita invariance}
\label{morita-sec}
In this section we prove Morita invariance results for reflexive homology. It is a remarkable property of Hochschild homology \cite[Theorem 3.7]{Dennis-Igusa} (see also \cite[Section 1.2]{Lod}), cyclic homology \cite[Corollary 1.7]{LQ} and dihedral homology \cite[Theorem 3.4]{KLS-dihed} that the given homology theory applied to the algebra of $(m\times m)$-matrices with entries in a $k$-algebra $A$, involutive in the case of dihedral homology, is isomorphic to the homology of $A$. We prove that reflexive homology shares this property. We also prove a more general result. We show that if two algebras are Hermitian Morita equivalent and satisfy a compatibility condition, then they have the same reflexive homology. It is worth remarking that Morita invariance is not a property shared by every homology theory associated to a crossed simplicial group. For example, it is shown in \cite[Remark 88]{Ault} and \cite[Corollary 5.11]{DMG-e-inf} that neither symmetric homology nor hyperoctahedral homology satisfy Morita invariance.

\subsection{Morita equivalence and Hermitian Morita equivalence}
We begin by recalling the definition of Morita equivalence for two (not necessarily involutive) $k$-algebras.

\begin{defn}
\label{Morita-defn}
Two unital $k$-algebras, $A$ and $B$, are said to be \emph{Morita equivalent} if there is an $A$-$B$-bimodule $P$ and a $B$-$A$-bimodule $Q$ together with an isomorphism of $A$-bimodules, $u\colon P\otimes_B Q \rightarrow A$, and an isomorphism of $B$-bimodules, $v\colon Q\otimes_{A} P\rightarrow B$.
\end{defn}

The concept of Morita equivalence can be extended to involutive $k$-algebras. This notion is known as \emph{Hermitian Morita equivalence}. This was introduced by Fr\"olich and McEvett \cite[Section 8]{F-Mc} (see also \cite{Hahn} and \cite[Section 2]{FS}).

\begin{defn}
\label{Hermitian-Morita-defn}
Let $A$ and $B$ be two unital, involutive $k$-algebras. We say that $A$ and $B$ are \emph{Hermitian Morita equivalent} if:
\begin{itemize}
\item $A$ and $B$ are Morita equivalent in the sense of Definition \ref{Morita-defn};
\item the isomorphisms $u$ and $v$ satisfy
\begin{itemize}
\item $u(p\otimes q)p^{\prime}= pv(q\otimes p^{\prime})$ and
\item $v(q\otimes p)q^{\prime}=qu(p\otimes q^{\prime})$
\end{itemize}
for all $p,\,p^{\prime}\in P$ and $q,\,q^{\prime}\in Q$;
\item there exists an additive bijection $\theta\colon P\rightarrow Q$ satisfying
\begin{itemize}
\item $\theta(apb)= \ol{b} \,\theta(p) \,\ol{a}$ for all $a \in A$, $b\in B$ and $p\in P$,
\item $u\left(p\otimes \theta(p^{\prime})\right) = \ol{u\left(p^{\prime}\otimes \theta(p)\right)}$ and
\item $v\left(\theta(p)\otimes p^{\prime}\right) = \ol{v\left(\theta(p^{\prime})\otimes p\right)}$.
\end{itemize}
\end{itemize}
\end{defn}

\begin{rem}
If follows from the fact that $u$ and $v$ are isomorphisms that there exist sets of elements in $P$, say $\llb p_1,\dotsc , p_l\rrb$ and $\llb p_1^{\prime} , \dotsc , p_m^{\prime}\rrb$, and sets of elements in $Q$, say $\llb q_1,\dotsc , q_l\rrb$ and $\llb q_1^{\prime} , \dotsc , q_m^{\prime}\rrb$, such that 
\[u\left(\sum_{i=1}^l p_j\otimes q_j\right) = 1_A \quad \text{and} \quad v\left(\sum_{j=1}^m q_k^{\prime}\otimes p_k^{\prime}\right) =1_B.\]
\end{rem}

\begin{defn}
\label{compatible-defn}
Let $A$ and $B$ be two unital, involutive $k$-algebras. Suppose that $A$ and $B$ are Hermitian Morita equivalent. Suppose that the additive bijection $\theta$ sends the set $\llb p_1,\dotsc , p_l\rrb$ to the set $\llb q_1,\dotsc , q_l\rrb$. In this case we say that $A$ and $B$ are \emph{compatible}.
\end{defn}

\subsection{The case of involutive matrix algebras} 
In this subsection we prove that reflexive homology satisfies Morita invariance for matrix algebras.

Let $A$ be an involutive $k$-algebra. The algebra of $(m\times m)$-matrices with entries in $A$, $\mathcal{M}_m(A)$, is an involutive $k$-algebra with involution defined by $\ol{\left(x_{ij}\right)}=\left(\ol{x_{ji}}\right)$. In other words, we take the transpose and apply the involution of $A$ entry-wise.
 
For $n\geqslant 0$, we define a $k$-module morphism
\[\mathrm{Tr}_n\colon \mathcal{M}_m(A)^{\otimes (n+1)}\rightarrow A^{\otimes (n+1)}\]
by
\[\mathrm{Tr}_n\left(X^{(0)}\otimes\cdots \otimes X^{(n)}\right) = \sum_{1\leqslant i_0,\dotsc , i_n\leqslant m} x_{i_0i_1}^{(0)}\otimes x_{i_1i_2}^{(1)}\otimes \cdots \otimes x_{i_ni_0}^{(n)}\]
where $x_{ij}^{(k)}$ is the element in the $i^{th}$ row of the $j^{th}$ column of the matrix $X^{(k)}$. As noted in \cite[Section 3]{KLS-dihed} this map is compatible with simplicial structure and the involution operators $r_n$ for $n\geqslant 0$. It follows that there are induced maps
\[\mathrm{Tr}_{\star}\colon HR_n^{\pm}\left(\mathcal{M}_m(A)\right) \rightarrow HR_n^{\pm}(A)\]
on reflexive homology for all $m\geqslant 1$ and $n\geqslant 0$.

\begin{thm}
\label{morita-thm}
The morphism
\[\mathrm{Tr}_{\star}\colon HR_n^{\pm}\left(\mathcal{M}_m(A)\right) \rightarrow HR_n^{\pm}(A)\]
on reflexive homology induced from the trace map is an isomorphism for all $m\geqslant 1$ and $n\geqslant 0$.
\end{thm}
\begin{proof}
By Proposition \ref{hyperhom-prop} we can consider $\mathrm{Tr}_{\star}$ as a morphism on hyperhomology. By the Morita invariance for Hochschild homology the trace maps induce isomorphisms $HH_{\star}\left(\mathcal{M}_m(A)\right) \rightarrow HH_{\star}(A)$.

Consider the bicomplexes $C_{\star,\star}\otimes_{\Delta R^{op}} \mathcal{L}^{\pm}\left(\mathcal{M}_m(A)\right)$ and $C_{\star,\star}\otimes_{\Delta R^{op}} \mathcal{L}^{\pm}\left(A\right)$. The isomorphisms $HH_{\star}\left(\mathcal{M}_m(A)\right) \rightarrow HH_{\star}(A)$ ensure that the vertical homology spectral sequences for these bicomplexes have isomorphic $E^1$-pages. The result now follows from the comparison theorem \cite[5.2.12]{weib}.
\end{proof} 

\subsection{A more general result}
In this section we prove a more general Morita invariance statement for reflexive homology. We show that if two algebras are Hermitian Morita equivalent and are compatible in the sense of Definition \ref{compatible-defn}, then they have the same reflexive homology.

\begin{defn}
\label{coeffs-involution-defn}
Let $A$ and $B$ be two unital, involutive $k$-algebras. Suppose that $A$ and $B$ are Hermitian Morita equivalent, so there exist bimodules $P$ and $Q$ and an additive bijection $\theta\colon P\rightarrow Q$ as in Definitions \ref{Morita-defn} and \ref{Hermitian-Morita-defn}. Let $M$ be an involutive $A$-bimodule. We define an involution the the $B$-bimodule $Q\otimes_A M\otimes_A P$ by
\[\ol{q\otimes m \otimes p}=\theta(p)\otimes \ol{m} \otimes \theta^{-1}(q).\] 
\end{defn}

\begin{thm}
Let $A$ and $B$ be two unital, involutive $k$-algebras. Suppose that they are Hermitian Morita equivalent and that they are compatible in the sense of Definition \ref{compatible-defn}. Let $M$ be an involutive $A$-bimodule. There exist natural isomorphisms of graded $k$-modules
\[HR_{\star}^{\pm}(A,M) \cong HR_{\star}^{\pm}(B, Q\otimes_A M \otimes_A P),\]
where $Q\otimes_A M \otimes_A P$ is equipped with the involution of Definition \ref{coeffs-involution-defn}.
\end{thm}
\begin{proof}
In \cite[1.2.7]{Lod}, Loday defines a chain homotopy equivalence
\[\psi_{\star}\colon C_{\star}(A,M)\rightarrow C_{\star}\left(B, Q\otimes_A M\otimes_A P\right)\]
between the Hochschild complexes. One can check that this is compatible with the involutions on both complexes, yielding a $C_2$-equivariant chain homotopy equivalence, from which the result follows.
\end{proof}

\section{Calculations}
\label{calc-sec}

In this section we provide some computations of reflexive homology. We can calculate the reflexive homology of the ground ring in two different ways; using the theory of crossed simplicial groups and direct computation from our bicomplex in the previous section. We demonstrate that our calculations agree with what is already known when working over a field of characteristic zero. We also give explicit descriptions of reflexive homology in degree zero for a commutative algebra.

\begin{prop}
There is an isomorphism of graded $k$-modules $HR_{\star}^{+}\left(k\right) \cong H_{\star}\left(BC_2 , k\right)$.
\end{prop}
\begin{proof}
Let $T\colon \Delta R^{op}\rightarrow \mathbf{Set}$ denote the trivial reflexive set. Explicitly, $T$ sends every object in $\Delta R^{op}$ to the one point set $\llb \ast\rrb$ with trivial involution. Using \cite[Corollary 6.13]{FL} we have
\[HR_{\star}^{+}(k) \cong HR_{\star}^{+}\left(k[T]\right) \cong H_{\star}\left(EC_2\times_{C_2} \left\lvert T\right\rvert, k\right) \cong H_{\star}\left(EC_2\times_{C_2} \ast , k\right) \cong H_{\star}\left(BC_2 , k\right)\]
as required.
\end{proof}

We can also obtain this calculation directly from the bicomplex defined in the previous section.

\begin{prop}
\label{degree-zero-prop}
Let $k$ be a commutative ring with trivial involution. Then
\begin{alignat*}{2}
     \begin{aligned}   
     HR_n^{+}(k) \cong \begin{cases}
k & n=0 \\
k/2k & n \text{ odd}\\
{}_2k & n>0 \text{ even}
\end{cases}
  \end{aligned}
   & \hskip 6em  &
  \begin{aligned}
   HR_n^{-}(k)\cong \begin{cases}
k/2k & n \text{ even}\\
{}_2k & n \text{ odd}
\end{cases}
  \end{aligned}
  \end{alignat*}
where ${}_2k$ denotes the $2$-torsion of $k$. 
\end{prop}
\begin{proof}
We will prove the result for $HR_n^{+}(k)$. The result for $HR_n^{-}(k)$ is similar.

Consider the bicomplex $C_{\star , \star}\otimes_{\Delta R^{op}} \mathcal{L}^{+}(k)$. Since the Hochschild homology of $k$ is isomorphic to $k$ concentrated in degree zero \cite[1.1.6]{Lod}, taking the vertical homology yields the complex
\[ 0\leftarrow k \xleftarrow{0} k \xleftarrow{2} k \xleftarrow{0} k \xleftarrow{2} k \xleftarrow{0} \cdots\]
in row zero. The result now follows by taking homology of this complex.
\end{proof} 

\begin{cor}
If $2$ is invertible in the ground ring then $HR_{\star}^{+}(k)$ is isomorphic to $k$ concentrated in degree zero and $HR_{\star}^{-}(k)$ is zero in all degrees.
\end{cor}
\begin{proof}
If $2$ is invertible in $k$ then the quotient module $k/2k$ and the $2$-torsion ${}_2k$ are both zero and the result follows from Proposition \ref{degree-zero-prop}.
\end{proof}

\begin{rem}
Recall from Remark \ref{LES-rem} that when we work over a field of characteristic zero, the direct sum of $HR_{\star}^{+}$ and $HR_{\star}^{-}$ is isomorphic to Hochschild homology, $HH_{\star}$. We note that our calculations for the ground ring $k$ agree with this. When $k$ is a field of characteristic zero we see that $HR_{\star}^{+}(k)\oplus HR_{\star}^{-}(k)$ is isomorphic to $k$ concentrated in degree zero, which is isomorphic to $HH_{\star}(k)$ \cite[1.1.6]{Lod}.
\end{rem}

\begin{thm}
Let $A$ be a commutative $k$-algebra.
\begin{itemize}
\item If $A$ has the trivial involution then $HR_0^{+}(A)\cong A$ and $HR_0^{-}(A)\cong A/2A$.
\item If $A$ has a non-trivial involution then $HR_0^{+}(A)$ is isomorphic to the coinvariants of $A$ under the involution.
\end{itemize}
\end{thm}
\begin{proof}
One can deduce these results directly from the bicomplex that computes reflexive homology.
\end{proof}

\section{Reflexive homology of a tensor algebra}
\label{tensor-alg-sec}
In this section we calculate the reflexive homology of a tensor algebra. Let $M$ be a $k$-module and consider a module automorphism $m\mapsto \overline{m}$ which squares to the identity. The identity automorphism is an example of such.

Consider the tensor algebra
\[TM = \bigoplus_{n=0}^{\infty} M^{\otimes n}\]
where $M^{\otimes 0}=k$. The product is given by concatenation, see \cite[A.1]{Lod} for instance. The tensor algebra has an involution determined on the $n^{th}$ summand by
\[r\left(m_1\otimes \cdots \otimes m_n\right) = \left(\overline{m_n} \otimes \overline{m_{n-1}}\otimes \cdots \otimes \overline{m_2}\otimes \ol{m_1}\right).\]
Note that the involution on $k$ is trivial.

Taking $A=TM$ in the definition of the Loday functor $\mathcal{L}^{+}(A)$, we obtain an induced involution on $\mathcal{L}^{+}(A)([n])=(TM)^{\otimes {n+1}}$ given by
\[r(a_0\otimes\cdots \otimes a_n) = \ol{a_0}\otimes \ol{a_n}\otimes \cdots \otimes \ol{a_1},\]
where each $A_i \in TM$.

\begin{thm}
\label{tensor-alg-thm}
Let $M$ be a $k$-module with a module automorphism $m\mapsto \overline{m}$ which squares to the identity. There is a natural isomorphism of $k$-modules 
\[HR_n^{+}(TM) \cong 
\begin{cases}
H_0\left(C_2, HH_0(TM)\right) & n=0\\
H_{n-1}\left(C_2, HH_1(TM)\right)\oplus H_n\left(C_2, HH_0(TM)\right) & n\geqslant 1.
\end{cases}
\]
\end{thm}
\begin{proof}
Consider the reflexive bicomplex for $TM$. By taking the vertical homology we obtain the $E^1$-page of a spectral sequence whose entries are the Hochschild homology groups of $TM$ with horizontal differentials induced from $1\pm r$.

Loday and Quillen \cite[Lemma 5.2]{LQ} have calculated the Hochschild homology of a tensor algebra as follows:
\[HH_n(TM) \cong
\begin{cases}
\bigoplus_{q\geqslant 0} \frac{M^{\otimes q}}{\llangle 1-t\rrangle} & n=0\\
\bigoplus_{q \geqslant 1} \left(M^{\otimes q}\right)^t & n=1\\
0 & \text{otherwise},
\end{cases}
\]
where $\left(M^{\otimes q}\right)^t$ are the invariants of the action of the cyclic group $C_q$ on $M^{\otimes q}$ and $M^{\otimes q}/\llangle 1-t\rrangle$ are the coinvariants.

Therefore the $E^1$-page of our spectral sequence is concentrated in rows zero and one as follows:
\begin{center}
\begin{tikzcd}
HH_1(TM)  & HH_1(TM)\arrow[l, "1+r",swap] & HH_1(TM)\arrow[l, "1-r",swap] & HH_1(TM)\arrow[l, "1+r",swap]& \cdots \arrow[l, "1-r",swap]\\
HH_0(TM)  & HH_0(TM)\arrow[l, "1-r", swap] & HH_0(TM)\arrow[l, "1+r",swap] & HH_0(TM)\arrow[l, "1-r", swap]& \cdots \arrow[l, "1+r",swap]
\end{tikzcd}
\end{center}
An argument similar to \cite[Lemma 2.2.1]{Lodder1} tells us that in row zero we have
\[r\left(m_1\otimes \cdots \otimes m_q\right) = \left(\overline{m_1} \otimes \overline{m_q}\otimes \cdots \otimes \overline{m_2}\right)\]
and in row one we have 
\[r\left(m_1\otimes \cdots \otimes m_q\right) = -\left(\ol{m_1}\otimes \ol{m_q}\otimes \cdots \otimes \ol{m_2}\right).\]
Furthermore, an argument similar to \cite[Lemma 2.3.2]{Lodder1} tells us that the differentials on the $E^2$-page are zero and so the spectral sequence collapses.

Taking homology on the $E^1$-page yields
\[E_{p,0}^2\cong H_p\left(C_2, HH_0(TM)\right) \quad \text{and} \quad E_{p,1}^2\cong H_p\left(C_2, HH_1(TM)\right)\]
with all other $E_{p,q}^2=0$, from which we can read off the result.
\end{proof}

\begin{rem}
\label{grading-rem}
We deduce from the theorem that the reflexive homology of a tensor algebra has a grading induced from the grading on the Hochschild homology. In homological degree zero, for $q\geqslant 0$, we have
\[HR_0^{+}\left(TM\right)_q = H_0\left(C_2,\frac{M^{\otimes q}}{\llangle 1-t\rrangle}\right).\]
Note that when $q=0$ we have
\[HR_0^{+}\left(TM\right)_0=H_0\left(C_2,k\right)\cong H_{0}(BC_2,k) \cong HR_{0}^{+}(k)\]
and when $q=1$ we have
\[HR_0^{+}\left(TM\right)_1=H_0\left(C_2,M\right).\]
In homological degree $p$ we have
\[HR_p^{+}(TM)_0 = H_p\left(C_2 , k\right) \cong H_p(BC_2,k) \cong HR_p^{+}(k)\]
and 
\[HR_p^{+}(TM)_q = H_p\left(C_2,\frac{M^{\otimes q}}{\llangle 1-t\rrangle}\right) \oplus H_{p-1}\left(C_2 , \left(M^{\otimes q}\right)^t\right)\]
for $q\geqslant 1$.

We note that, as mentioned above, our grading on the reflexive homology of a tensor algebra is induced from the grading on the Hochschild homology of a tensor algebra. For cyclic homology \cite[Proposition 5.4]{LQ} and dihedral homology \cite[Theorem 2.1.1]{Lodder1} there is a grading induced directly from the grading on a tensor algebra. In both of these cases this follows from analysis of the norm map $N$ (see \cite[2.1.0]{Lod} for instance) which we do not have in the reflexive case.
\end{rem}

\section{Reflexive homology of a group algebra}
\label{kg-sec}
In this section we will study the reflexive homology of the group algebra of a discrete group. We will recall simplicial models for the bar construction and for the classifying space on a group. We will show that these extend to reflexive sets. We will show that the reflexive homology of a group algebra is isomorphic to the $C_2$-equivariant homology of the free loop space on $BG$. As a consequence we will deduce the analogous result for dihedral homology, namely that the dihedral homology of a group algebra is isomorphic to the $O(2)$-equivariant homology of the free loop space on $BG$. We will then show that we can decompose the reflexive homology of a group algebra in terms of the conjugacy classes of the group.

These results fit into a broader story of using the homology theories associated to crossed simplicial groups to calculate interesting information about loop spaces. Hochschild homology and cyclic homology of $k[G]$ are known to coincide with the homologies of $\mathcal{L}BG$ and the $S^1$-equivariant Borel construction on  $\mathcal{L}BG$ respectively \cite[7.3.13]{Lod}. The symmetric and hyperoctahedral theories are known to compute the homology and $C_2$-equivariant homology of certain infinite loop spaces on $BG$, see \cite[Corollary 40]{Ault} and \cite[Theorem 8.8]{DMG-hyp}.

The decomposition we provide fits into a bigger picture of decomposing homology theories associated to crossed simplicial groups in the case of a group algebra. Results of this form were proved by Burghelea \cite[Theorem 1]{DB-cyc} for Hochschild and cyclic homology (see also \cite[Theorem 7.4.6]{Lod}) and by Loday \cite[Proposition 4.9]{Lod-dihed} for dihedral homology. 

\subsection{Bar construction and classifying spaces}
We recall some simplicial models for the bar construction \cite[7.3.10]{Lod} and classifying space \cite[B.12]{Lod} of a group and extend them to reflexive sets.

\begin{defn}
Let $G$ be a discrete group. For $n\geqslant 0$, let $\Gamma_nG=G^{n+1}$, the $(n+1)$-fold Cartesian product. The face maps are defined by
\[\partial_i\left(g_0,\dotsc , g_n\right)=\begin{cases}
\left(g_0,\dotsc ,g_{i-1}, g_ig_{i+1}, g_{i+2},\dotsc , g_n\right) & 0\leqslant i \leqslant n-1\\
\left(g_ng_0, g_1, \dotsc , g_{n-1}\right) & i=n.
\end{cases}
\]
The degeneracy maps insert the identity element of $G$ into the tuple. We extend this to a reflexive set by defining
\[r_n\left(g_0,\dotsc , g_n\right)= \left(g_0^{-1}, g_n^{-1},\dotsc , g_1^{-1}\right).\]
\end{defn}

\begin{rem}
\label{dihedral-rem}
Loday has already shown that $\Gamma_{\star}G$ is a cyclic set, by defining 
\[t_n\left(g_0,\dotsc , g_n\right) = \left(g_n, g_0,\dotsc , g_{n-1}\right).\]
One can easily check that the reflexive structure and the cyclic structure are compatible, giving $\Gamma_{\star}G$ the structure of a dihedral set. The fact that the reflexive structure that we have defined is compatible with the cyclic structure of $\Gamma_{\star}G$ is key to proving Theorem \ref{grp-alg-thm}.
\end{rem}

\begin{defn}
Let $B_nG=G^{n}$ for $n\geqslant 0$. The face maps are given by
\[\partial_i\left(g_1,\dotsc , g_n\right) = 
\begin{cases}
\left(g_2, \dotsc , g_n\right) & i=0\\
\left(g_1 , \dotsc , g_ig_{i+1}, \dotsc , g_n\right) & 1\leqslant i \leqslant n\\
\left(g_1,\dotsc , g_{n-1}\right) & i=n.
\end{cases}
\]
The degeneracy maps insert the identity element into the tuple. We can extend this to a reflexive set by defining 
\[r_n\left(g_1,\dotsc , g_n\right) = \left(g_n^{-1},\dotsc , g_1^{-1}\right).\]
\end{defn}

\begin{defn}
Let $G$ be a discrete group. We define the \emph{reflexive homology of $G$} to be
\[HR_{\star}^{+}\left(G,k\right)\coloneqq HR_{\star}^{+}\left(k\left[B_{\ast}G\right]\right).\]
\end{defn}

\begin{rem}
\label{proj-map-rem}
Note that we have a projection map, which is a map of reflexive sets, 
\[p\colon \Gamma_{\star}G \rightarrow B_{\star}G,\]
determined in degree $n$ by $\left(g_0,\dotsc , g_n\right) \mapsto \left(g_1,\dotsc, g_n\right)$.
\end{rem}

\begin{rem}
As for $\Gamma_{\star}G$, $B_{\star}G$ has a cyclic structure. This is given by 
\[t_n\left(g_1,\dotsc ,g_n\right)= \left(\left(g_1\cdots g_n\right)^{-1}, g_1,\dotsc , g_{n-1}\right),\]
as described in \cite[7.3.3]{Lod} (with $z=1$) for example. This is compatible with the reflexive structure on $B_{\star}G$, as described in \cite[2.3]{Dunn}, giving $B_{\star}G$ the structure of a dihedral set.
\end{rem}

\subsection{Reflexive homology of a group algebra}
Let $G$ be a discrete group. In this subsection we prove that the reflexive homology of a group algebra $k[G]$ is isomorphic to the homology of the $C_2$-equivariant Borel construction on the free loop space of $BG$.

\begin{thm}
\label{grp-alg-thm}
Let $G$ be a discrete group. Let $k[G]$ denote its group algebra and let $BG$ be its classifying space. Let $\mathcal{L}BG$ denote the free loop space on $BG$. There is an isomorphism of graded $k$-modules
\[HR_{\star}^{+}\left(k[G]\right) \cong H_{\star}\left(EC_2 \times_{C_2} \mathcal{L}BG , k \right),\]
where the $C_2$-action on $\mathcal{L}BG$ is induced from the reflexive structure of $B_{\star}G$ and reversing the direction of loops.
\end{thm}
\begin{proof}
In \cite[7.3.11]{Lod}, Loday shows that there is a homotopy equivalence $\gamma \colon \llv \Gamma_{\star}G\rrv \rightarrow \mathcal{L}BG$. This is done by considering the adjoint functors $S^1\times -$ and $\mathcal{L}(-)$. In particular, there is a map
\[S^1 \times \llv \Gamma_{\star}G\rrv \rightarrow \llv  \Gamma_{\star}G\rrv  \rightarrow \llv B_{\star}G\rrv\]
where the first map is induced from the cyclic structure of $\Gamma_{\star}G$, as described in \cite[Section 5]{FL}, and the second is induced from the projection map $p\colon \Gamma_{\star}G \rightarrow B_{\star}G$ described in Remark \ref{proj-map-rem}. The adjoint map is $\gamma$, the necessary homotopy equivalence. 

We extend this to a $C_2$-equivariant homotopy equivalence. As noted in Remark \ref{dihedral-rem}, the reflexive structure of $\Gamma_{\star}G$ is compatible with the cyclic structure. Explicitly, the circle $S^1$ is the geometric realization of the cyclic crossed simplicial group, $\llb C_n \rrb$, by \cite[6.3.6]{Lod}. The level-wise involution determined by sending the generator of the cyclic group to its inverse induces a $C_2$-action on the geometric realization. Similarly, the reflexive structures on $\Gamma_{\star}G$ and $B_{\star}G$ induce $C_2$-actions on the realizations. Note that the involution on geometric realizations also flips the simplex. One easily checks that with these definitions the map
\[S^1 \times \llv \Gamma_{\star}G\rrv   \rightarrow \llv B_{\star}G\rrv\]
is $C_2$-equivariant. Under the adjunction between $S^1 \times -$ and $\mathcal{L}(-)$ this yields that
\[\gamma \colon \llv \Gamma_{\star}G\rrv \rightarrow \mathcal{L}BG\]
is a $C_2$-equivariant map, with the involution on $\mathcal{L}BG$ given by applying the involution on $BG$ and reversing the direction of loops. Combining this with \cite[7.3.11]{Lod}, we see that $\gamma$ is a $C_2$-equivariant homotopy equivalence.

Using \cite[7.3.13]{Lod} and \cite[6.13]{FL}, we have
\[HR_{\star}^{+}\left(k[G]\right) \cong H_{\star}\left(EC_2 \times_{C_2} \llv\Gamma_{\star}G\rrv \right) \cong H_{\star}\left(EC_2\times_{C_2} \mathcal{L}BG\right)\]
as required.
\end{proof}

Using the fact that the reflexive structure that we have described is compatible with the cyclic structure of \cite[7.3.11]{Lod} we can deduce the analogous result for dihedral homology.

\begin{thm}
Let $G$ be a discrete group. Let $k[G]$ denote its group algebra and let $BG$ be its classifying space. Let $\mathcal{L}BG$ denote the free loop space on $BG$. There is an isomorphism of graded $k$-modules
\[HD_{\star}^{+}\left(k[G]\right) \cong H_{\star}\left(EO(2) \times_{O(2)} \mathcal{L}BG , k \right),\]
where the $O(2)$-action on $\mathcal{L}BG$ is induced from the dihedral structure of $B_{\star}G$ and the action of $O(2)$ on loops.
\end{thm}
\begin{proof}
Since the reflexive structure of $\Gamma_{\star}G$ and $B_{\star}G$ used in Theorem \ref{grp-alg-thm} is compatible with the cyclic structures used in \cite[7.3.11]{Lod}, we can deduce that $\gamma$ is an $O(2)$-equivariant map, from which the result follows.
\end{proof}

\subsection{Decomposition by conjugacy classes}
In this subsection we prove that the reflexive homology of the group algebra of a discrete group can be decomposed into a direct sum indexed by the conjugacy classes of the group, where the summands are defined in terms of a reflexive analogue of group homology. We prove that under nice conditions, for example when the group is abelian, we can give a simple description of the summands.

\begin{defn}
Let $G$ be a discrete group.
\begin{itemize}
\item Let $\overline{k[G]}$ denote $k[G]$ considered as a right $k[G]$-module with the action $h\ast g = g^{-1}hg$.
\item Let $\llangle G\rrangle$ denote the set of conjugacy classes of $G$. We will choose a representative $z$ in each class and denote the class by $\llangle z \rrangle$.
\item For $z\in G$ we denote by $G_z$ the centralizer of $z$ in $G$, that is, $G_z=\llb g\in G : gz=zg\rrb$.
\end{itemize}
\end{defn}

\begin{defn}
For a group $G$ and a right $k[G]$-module $M$ we denote by $C_{\star}(G,M)$ the \emph{Eilenberg-Mac Lane complex} \cite[C.2, C.3]{Lod}. We will denote an element in degree $n$ by $h\otimes \left[g_1, \dotsc ,g_n\right]$.
\end{defn}

In order to prove our decomposition of the reflexive homology of a group algebra we need to recall certain reflexive sets. The involutions defined here can be found in the proof of \cite[Proposition 4.9]{Lod-dihed}.

\begin{defn}
\label{action1}
Let $G$ be a discrete group.
\begin{itemize}
\item We define the action of $r_{n}$ on $k[G]^{\otimes n+1}$ to be determined by 
\[ r_{n}\left(g_0,\dotsc , g_n\right) = (-1)^{n(n+1)/2} \left(g_0^{-1}, g_n^{-1}, \dotsc , g_1^{-1}\right).\]
\item Let $M=\overline{k[G]}$. The action of $r_n$ on $C_{n}\left(G,M\right)$ is given by
\[ r_{n} \left(h \otimes \left[g_1,\dotsc , g_n\right]\right) = (-1)^{n(n+1)/2}\left( g^{-1} h^{-1} g \otimes \left[g_n^{-1}, \dotsc , g_1^{-1}\right]\right)\]
where $g=g_1\cdots g_n$.
\item We define the action of $r_n$ on 
\[\bigoplus_{\llangle z \rrangle \in \llangle G\rrangle} C_n\left(G, k\left[ G/G_z\right]\right)\]
by
\[ r_{n} \left(h \otimes \left[g_1,\dotsc , g_n\right]\right) = (-1)^{n(n+1)/2}\left( g^{-1} h^{-1} g \otimes \left[g_n^{-1}, \dotsc , g_1^{-1}\right]\right)\]
where $g=g_1\cdots g_n$.
\end{itemize}
\end{defn}

\begin{thm}
\label{kg-thm}
Let $G$ be a discrete group. For each $n\geqslant 0$ there is an isomorphism of $k$-modules
\[HR_n^{+}\left(k[G]\right) \cong \bigoplus_{\llangle z \rrangle \in \llangle G\rrangle} HR_n^{+}\left(G , k[G/G_z]\right).\]
\end{thm}
\begin{proof}
Let $M=\overline{k[G]}$. Recall the Mac Lane isomorphism
\[\Phi \colon k[G]^{n+1}\rightarrow C_n\left(G, M\right)\]
from \cite[Section 7.4]{Lod} for example. The Mac Lane isomorphism is compatible with the actions described in Definition \ref{action1}. Recall from the proof of \cite[Proposition 4.7]{Lod-dihed} that there is an isomorphism of right $k[G]$-modules
\[\xi \colon \overline{k[G]}\rightarrow \bigoplus_{\llangle z \rrangle \in \llangle G\rrangle} k\left[ G/G_z\right].\]
This isomorphism is also compatible with the actions described in Definition \ref{action1}

We therefore have isomorphisms
\[k[G]^{n+1} \xrightarrow{\Phi} C_n\left(G, M\right) \xrightarrow{\xi} \bigoplus_{\llangle z \rrangle \in \llangle G\rrangle} C_{n}\left(G,k[G/G_z]\right)\]
compatible with the reflexive $k$-module structures and the result follows upon taking reflexive homology.
\end{proof}

\begin{cor}
If $G$ is an abelian group then 
\[HR_{\star}^{+}\left(k[G]\right) \cong \bigoplus_{\left\lvert G\right\rvert} HR_{\star}^{+}\left(G,k\right).\]
\end{cor}
\begin{proof}
If $G$ is abelian, then each element $z\in G$ has its own conjugacy class and the centralizer $G_z$ is isomorphic to $G$. 
\end{proof} 

For non-abelian groups we can also identify some of the summands in the decomposition.

\begin{prop}
Let $G$ be a discrete group. The $\llangle 1\rrangle$-component of $HR_{\star}^{+}\left(k[G]\right)$ is isomorphic to $HR_{\star}^{+}\left(G,k\right)$.
\end{prop}
\begin{proof}
The centralizer $G_1$ is equal to $G$ so the reflexive homology of $C_{\star}\left(G , k[G/G_1]\right)$ is the reflexive homology of $C_{\star}(G,k)$.
\end{proof}

\begin{prop}
Let $G$ be a discrete group. Let $z \in G$ be a central element of order two. The $\llangle z\rrangle$-component of $HR_{\star}^{+}\left(k[G]\right)$ is isomorphic to $HR_{\star}^{+}\left(G,k\right)$.
\end{prop}
\begin{proof}
Recall from the proof of \cite[Proposition 4.7]{Lod-dihed} that there is a quasi-isomorphism
\[k\otimes C_n\left(G_z\right)\rightarrow k\left[G/G_z\right] \otimes C_n(G)\]
for any $z\in G$.
When $z$ is a central element of order two this quasi-isomorphism is compatible with the reflexive structure described in Definition \ref{action1}, whence the result.
\end{proof}

\section{Reflexive homology of singular chains on a Moore loop space}
\label{Moore-loop-sec}
Goodwillie \cite[Section \RNum{5}]{Goodwillie} proved that for a sufficiently nice space $X$, the cyclic homology of the singular chain complex of the Moore loop space of $X$ is isomorphic to the $S^1$-equivariant homology of the free loop space on $X$. Dunn \cite[3.6]{Dunn} proved that this result can be extended to a result for the dihedral homology theory and $O(2)$-equivariant homology. In this section we will prove the analogous theorem for the reflexive homology theory and $C_2$-equivariant homology. We will also provide a $C_2$-equivariant analogue of Lodder's result for free-loop suspension spaces \cite[3.3.3]{Lodder1}. 

\subsection{Locally equiconnected spaces}
In this brief subsection we recall what it means for a space to be \emph{locally equiconnected}, commonly abbreviated to LEC. Our results rely on the work of Dunn \cite{Dunn}, which requires this assumption in relation to Milnor's simplicial group model for loop-suspension spaces (see \cite{Milnor} and \cite{Lewis}). Local equiconnectivity was introduced by Fox \cite{Fox} as a strengthened form of local contractibilty and a weakened form of the absolute neighbourhood retract property. There are several equivalent definitions but we will use the definition in terms of \emph{Hurewicz cofibrations}. The reader is directed to \cite{Strom}, where the definition originally occurred, and \cite[Chapter \RNum{7}]{Bredon}, for a textbook account of Hurewicz cofibrations. 

\begin{defn}
A topological space is said to be \emph{locally equiconnected} if the diagonal map $X\rightarrow X \times X$ is a Hurewicz cofibration.
\end{defn}

\begin{eg}
The notion of LEC space also appears in work of Serre, where it is called ULC \cite[p. 490]{Serre}. All CW-complexes are LEC spaces \cite[Corollary \RNum{3}.2]{Dyer-Eilenberg}. 
\end{eg}

\subsection{Moore loops and free loop spaces}
\label{moore-loop-free-loop-subsec}

We begin by recalling the definition of the Moore loop space.

\begin{defn}
Let $M\colon \mathbf{Top}_{\star}\rightarrow \mathbf{Top}_{\star}$ denote the Moore loop functor. For $X\in \mathbf{Top_{\star}}$ with basepoint $\star$, the topological space $M(X)$ is the subset of 
\[\mathrm{Hom}_{\mathbf{Top}_{\star}}\left([0,\infty),X\right)\times [0,\infty)\]
consisting of all pairs of the form $\left( f, r\right)$ such that $f(t)=\star$ for all $t\geqslant r$. In general we will omit the $r$ and refer to a Moore loop $f$.
\end{defn} 

Let $Y$ be a group-like LEC topological monoid with involution. As noted in Example \ref{sing-chain-eg2}, the singular chain complex $S_{\star}\left(Y, k\right)$ is an involutive DGA, with the involution induced from $Y$. Recall the reflexive bar construction, $\Gamma\left(S_{\star}\left(Y, k\right)\right)$, on $S_{\star}\left(Y, k\right)$ from Example \ref{sing-chain-eg2}.

\begin{thm}
\label{moore-loop-free-loop-thm}
Let $Y$ be a group-like LEC topological monoid with involution. There is an isomorphism of graded $k$-modules
\[HR_{\ast}\left(\Gamma\left(S_{\star}\left(Y, k\right)\right)\right)\cong H_{\ast}\left(EC_2 \times_{C_2} \mathcal{L}BY,k\right)\]
where the $C_2$-action on $\mathcal{L}BY$ is induced from the involution on $Y$ and reversing the direction of loops.
\end{thm}
\begin{proof}
This follows directly from \cite[3.6]{Dunn}. We note that the given theorem relies on Dunn's Propositions 2.10, 3.2, 3.3 and 3.5. One can check that these results hold when we only consider the simplicial and reflexive structure from the dihedral objects considered in that paper. 
\end{proof}

\begin{cor}
Let $X$ be a connected, LEC topological space. There is an isomorphism of graded $k$-modules
\[HR_{\ast}\left(\Gamma\left(S_{\star}\left(MX, k\right)\right)\right)\cong H_{\ast}\left(EC_2 \times_{C_2} \mathcal{L}X,k\right),\]
where $X$ is equipped with the trivial involution and $MX$ is equipped with the involution which reverses the direction of loops.
\end{cor}
\begin{proof}
May \cite[15.4]{may-fibrations} shows that there is a map $\xi\colon BMX\rightarrow X$, which is a weak homotopy equivalence if $X$ is connected. As noted in \cite[2.9]{Dunn}, with the given involutions the map $\xi$ is $C_2$-equivariant. The corollary now follows from Theorem \ref{moore-loop-free-loop-thm} by setting $Y=MX$.
\end{proof}

\subsection{Moore loops and free loop-suspension spaces}
In this section we prove a reflexive analogue of Lodder's result for free loop-suspension spaces \cite[3.3.3]{Lodder1}. We identify the $C_2$-equivariant homology of a free loop-suspension space $\mathcal{L}\Sigma X$ as the reflexive homology of the singular chains on a reflexive space constructed from the Moore loop space of $X$. In this case we can use a smaller construction than the reflexive bar construction on a DGA since the suspension of a topological space is always path-connected.

We begin by recalling the definition of the suspension functor.

\begin{defn}
Let $\Sigma\colon \mathbf{Top}_{\star} \rightarrow \mathbf{Top}_{\star}$ denote the suspension functor $S^1 \wedge -$.
\end{defn}

\begin{defn}
Let $X$ be a based topological space. Let $\Sigma X=S^1 \wedge X$ be equipped with the involution $\overline{(t,x)}=(1-t,x)$, given by reversing the suspension co-ordinate. We define a reflexive topological space
\[M\Sigma X(-)\colon \Delta R^{op} \rightarrow \mathbf{Top}_{\star}\]
as follows.

For each $n\geqslant 0$, we have $M\Sigma X([n])=\left(M\Sigma X\right)^{n+1}$, the $(n+1)$-fold Cartesian product of the Moore loop space on $\Sigma X$.

The face maps $\partial_i\colon M\Sigma X([n])\rightarrow M\Sigma X([n-1])$ are given by concatenation of loops:
\[\partial_i\left(f_0, \dotsc , f_n\right) = \begin{cases}
\left( f_0,\dotsc , f_{i-1}, f_i\cdot f_{i+1} , f_{i+2},\dotsc , f_n\right) & 0\leqslant i \leqslant n-1\\
\left(f_n\cdot f_0, f_1,\dotsc , f_{n-1}\right) & i=n.
\end{cases}
\]
The degeneracy maps $s_j\colon M\Sigma X([n])\rightarrow M\Sigma X([n+1])$ are given by inserting the trivial loop:
\[s_j\left(f_0, \dotsc , f_n\right)= \left(f_0, \dotsc , f_j , 1 , f_{j+1},\dotsc , f_n\right)\]
for $0\leqslant j \leqslant n$.

The reflexive operators $r_n\colon M\Sigma X([n]) \rightarrow M\Sigma X([n])$ for $n\geqslant 0$ are given by reversing the direction of loops and applying the involution on $\Sigma X$: 
\[r_n\left(f_0, \dotsc , f_n\right) = \left(\overline{f}_0^{-1}, \overline{f}_n^{-1}, \dotsc , \overline{f}_1^{-1}\right).\]
\end{defn}

\begin{defn}
Let $\mathcal{S}$ denote the reflexive chain complex obtained by taking the singular chain complex $S_{\star}\left(M\Sigma X(-), k\right)$ on the reflexive topological space $M\Sigma X(-)$.
\end{defn}

\begin{thm}
\label{freeloop-susp-thm}
Let $X$ be a based topological space. Let $\mathcal{S}$ denote the singular chain complex on the reflexive topological space $M\Sigma X(-)$. There is an isomorphism of graded $k$-modules
\[HR_{\ast}\left(\mathcal{S}\right) \cong H_{\ast}\left( EC_2 \times_{C_2} \mathcal{L}\Sigma X,k\right).\]
\end{thm}
\begin{proof}
By combining the constructions in \cite[Appendix B]{Lod} for singular homology with results on homotopy colimits for crossed simplicial groups \cite[Section 6]{FL} we see that given a reflexive space $Y_{\star}$, there are isomorphisms
\[HR_{n}\left(\mathcal{S}(Y_{\star})\right) \cong H_n\left(\mathrm{hocolim}_{\Delta R^{op}}\left(Y_{\star}\right)\right)\cong H_n\left(EC_2\times_{C_2} \llv Y_{\star}\rrv\right),\]
where the $C_2$-action on $\llv Y_{\star}\rrv$ is induced from the reflexive structure.

Goodwillie \cite[Section \RNum{5}]{Goodwillie} has constructed a weak equivalence $\llv M\Sigma X(-)\rrv \rightarrow \mathcal{L}\Sigma X$. By incorporating the reflexive action defined in \cite[Section 3.4]{Lodder1} we extend this to a $C_2$-equivariant weak equivalence. The proof of this fact follows by mimicking \cite[3.4.1]{Lodder1}, replacing the dihedral crossed simplicial group with the reflexive crossed simplicial group.

We therefore have a weak equivalence 
\[EC_2\times_{C_2} \llv M\Sigma X(-)\rrv \rightarrow EC_2 \times_{C_2}\mathcal{L}\Sigma X\] which yields
\[HR_{\ast}\left(\mathcal{S}\right)\cong H_{\ast}\left(EC_2\times_{C_2} \llv M\Sigma X(-)\rrv\right) \cong H_{\ast}\left( EC_2 \times_{C_2} \mathcal{L}\Sigma X,k\right) \]
as required.
\end{proof}

\section{Relationship to involutive Hochschild homology}
\label{ihh-sec}

Involutive Hochschild homology was introduced by Braun \cite{Braun} to study involutive algebras and involutive $A_{\infty}$-algebras. Fern\`andez-Val\`encia and Giansiracusa \cite{RFG} developed the homological algebra of involutive Hochschild homology. In particular, they show that  involutive Hochschild homology can be expressed as $\mathrm{Tor}$ over an involutive version of the enveloping algebra. The original motivation for constructing involutive Hochschild homology was to extend work of Costello \cite{Costello} to study the connection between unoriented two-dimensional topological conformal field theories and involutive Calabi-Yau $A_{\infty}$-algebras \cite{Braun}, \cite{FV}. Involutive Hochschild homology also arises in the study of the (co)homology of involutive dendriform algebras \cite{DasSaha}. We show that under certain conditions involutive Hochschild homology coincides with reflexive homology.

Let $k$ be a field. Let $A$ be an involutive $k$-algebra and let $A^e=A\otimes A^{op}$ denote the enveloping algebra. For this section only, let $C_2=\left\langle t\mid t^2=1\right\rangle$. The group $C_2$ acts on $A^e$ by the rule $t\left(a_1\otimes a_2\right)=\ol{a_2} \otimes \ol{a_1}$. Fern\`andez-Val\`encia and Giansiracusa define the \emph{involutive enveloping algebra} $A^{ie}$ to be $A^e\otimes k[C_2]$ with the product determined by
\[ \left(a_1\otimes t^i\right) \left(a_2\otimes t^j\right)= \left(a_1 \cdot t^i(a_2)\right)\otimes t^{i+j}.\] 
The \emph{involutive Hochschild homology of $A$ with coefficients in an involutive $A$-bimodule $M$}, denoted by $iHH_{\star}(A,M)$, is defined to be $\mathrm{Tor}_{\star}^{A^{ie}}\left(A,M\right)$ \cite[3.3.1]{RFG}.
Recall that an involutive vector space is projective if, when viewed as a $k[C_2]$-module, it is a direct summand of a free module.

\begin{thm}
Let $k$ be a field of characteristic zero. Let $A$ be a projective involutive $k$-algebra and let $M$ be an involutive $A$-bimodule. There exists an isomorphism of graded $k$-modules
\[HR_{\star}^{+}(A,M)\cong \mathrm{Tor}_{\star}^{A^{ie}}\left(A,M\right).\]
\end{thm}
\begin{proof}
Proposition \ref{char-zero-prop} and \cite[3.3.2]{RFG} tell us that, under these conditions, both $HR_{\star}^{+}(A,M)$ and $iHH_{\star}(A,M)$ are isomorphic to the homology of the quotient of the Hochschild complex by the involution action, from which the result follows.  
\end{proof}

\section{Reflexive structure in real topological Hochschild homology}
\label{THR-section}

The structure of the reflexive crossed simplicial group arises implicitly in the literature on real topological Hochschild homology, although the category $\Delta R$ very rarely explicitly arises. In this expository section we explain how certain constructions in real topological Hochschild homology can be stated in terms of the reflexive crossed simplicial group.

We begin by recalling the definition of a real simplicial object from \cite[1.4.1]{dotto-thesis}, \cite[1.2]{DMPR}, \cite[1.1]{Hogenhaven2}, \cite[2.3]{AKGH}.

\begin{defn}
A \emph{real simplicial object} $X$ in a category $\mathbf{C}$ is a functor $X\colon \Delta^{op}\rightarrow \mathbf{C}$ together with maps $\omega_n\colon X([n])\rightarrow X([n])$ such that $\omega_n^2=id_{X([n])}$ and
\begin{itemize}
\item $d_i\circ \omega_n = \omega_{n-1}d_{n-i}$,
\item $s_i\circ \omega_n = \omega_{n+1}s_{n-i}$.
\end{itemize}
\end{defn}

\begin{rem}
By comparing with Definition \ref{delta-r-defn} we see that a real simplical object $X$ in a category $\mathbf{C}$ is precisely a functor $X\colon \Delta R^{op} \rightarrow \mathbf{C}$. In other words, it is the same thing as a reflexive object in $\mathbf{C}$. We note that the category $\Delta R$ does appear in \cite[1.2]{DMPR} with the notation $\Delta_R^{op}$.
\end{rem}

\begin{eg}
If we take $\mathbf{C}=\mathbf{Set}$, we obtain the category of \emph{reflexive sets} or \emph{real simplicial sets}. As noted in the introduction, Spali\'nski \cite[Section 3]{spal-discrete} has shown that the category of reflexive sets admits a model structure that is Quillen equivalent to the category of $C_2$-spaces with the fixed-point model structure. Furthermore, reflexive sets play an important role in the homotopy theory of dihedral sets (see \cite{spalinski2} and \cite{spal-discrete}).

It is worth remarking that, using the crossed simplical group structure, we can use a result of Fiedorowicz and Loday \cite[6.13]{FL} to calculate the homology of $C_2$-equivariant Borel constructions. Let $X$ be a reflexive set and let $k[X]$ be the composition with the free $k$-module functor. There is an isomorphism of graded $k$-modules $HR_{\star}^{+}(k[X]) \cong H_{\star}\left(EC_2 \times_{C_2} \llv X\rrv , k\right)$. 
\end{eg}

\begin{eg}
Let $\mathbf{C}=\mathbf{Sp}^{O}$, the category of orthogonal ring spectra as introduced in \cite{MMSS}. Let $(A,w)$ be an orthogonal ring spectrum with anti-involution in the sense of \cite[Definition 2.1]{DMPR} and let $(M,j)$ be an $A$-bimodule in the sense of Definition 2.5 of the same paper. The \emph{dihedral nerve} $N_{\wedge}^{di}(A,M)$ of \cite[Definition 2.9]{DMPR} is a functor $\Delta R^{op} \rightarrow\mathbf{Sp}^{O}$. In other words, it is a reflexive object in the category of orthogonal spectra or a real simplicial orthogonal spectrum. We note that this is called \emph{dihedral} because, in the case where we take $M=A$, there is also an action of the cyclic groups, giving the structure of a functor $\Delta D^{op}\rightarrow\mathbf{Sp}^{O}$.

The dihedral nerve plays an important role in the construction of $\mathrm{THR}_{\bullet}(A;M)$ (see \cite[Definition 2.18]{DMPR}) whose geometric realization is the \emph{real topological Hochschild homology of $A$ with coefficients in $M$}. We note that $\mathrm{THR}_{\bullet}(A;M)$ is itself a real simplicial orthogonal spectrum and therefore can be considered as a functor $\Delta R^{op}\rightarrow \mathbf{Sp}^{O}$.
\end{eg}

\begin{rem}
Our reflexive homology theory, as defined in Section \ref{ref-hom-sec}, takes as input a $k$-algebra with involution and gives a graded $k$-module as output. In recent work, Angelini-Knoll, Gerhardt and Hill \cite{AKGH} have introduced another theory for rings with involution, called \emph{real Hochschild homology}. This theory takes as input the Mackey functor associated to a ring with involution (see \cite[6.12]{AKGH}) and gives graded equivariant Mackey functor as output (see \cite[6.15]{AKGH}). They also prove that their theory is related to real topological Hochschild homology via a linearization map (\cite[6.20]{AKGH}).
\end{rem}


\begin{thebibliography}{10}

\bibitem{AKGH}
Gabriel Angelini-Knoll, Teena Gerhardt, and Michael Hill.
\newblock Real topological {H}ochschild homology via the norm and {R}eal {W}itt
  vectors, 2021.
\newblock arXiv e-print 2111.06970.

\bibitem{Ault}
Shaun~V. Ault.
\newblock Symmetric homology of algebras.
\newblock {\em Algebr. Geom. Topol.}, 10(4):2343--2408, 2010.

\bibitem{BCS}
Andrew~J. Blumberg, Ralph~L. Cohen, and Christian Schlichtkrull.
\newblock Topological {H}ochschild homology of {T}hom spectra and the free loop
  space.
\newblock {\em Geom. Topol.}, 14(2):1165--1242, 2010.

\bibitem{Braun}
Christopher {Braun}.
\newblock {Involutive \(A_\infty\)-algebras and dihedral cohomology.}
\newblock {\em {J. Homotopy Relat. Struct.}}, 9(2):317--337, 2014.

\bibitem{Bredon}
Glen~E. Bredon.
\newblock {\em Topology and geometry}, volume 139 of {\em Graduate Texts in
  Mathematics}.
\newblock Springer-Verlag, New York, 1993.

\bibitem{Burg-Fie}
D.~Burghelea and Z.~Fiedorowicz.
\newblock Cyclic homology and algebraic {K}-theory of spaces. {II}.
\newblock {\em Topology}, 25:303--317, 1986.

\bibitem{DB-cyc}
Dan Burghelea.
\newblock The cyclic homology of the group rings.
\newblock {\em Comment. Math. Helv.}, 60:354--365, 1985.

\bibitem{HKT1}
Baptiste Calm\`es, Emanuele Dotto, Yonatan Harpaz, Fabian Hebestreit, Markus
  Land, Kristian Moi, Denis Nardin, Thomas Nikolaus, and Wolfgang Steimle.
\newblock Hermitian {K}-theory for stable {$\infty$}-categories {I}:
  {F}oundations.
\newblock {\em Selecta Math. (N.S.)}, 29(1):Paper No. 10, 269, 2023.

\bibitem{HKT2}
Baptiste Calmès, Emanuele Dotto, Yonatan Harpaz, Fabian Hebestreit, Markus
  Land, Kristian Moi, Denis Nardin, Thomas Nikolaus, and Wolfgang Steimle.
\newblock Hermitian {K}-theory for stable $\infty$-categories {II}: Cobordism
  categories and additivity, 2021.
\newblock URL: \url{https://arxiv.org/abs/2009.07224}.

\bibitem{HKT3}
Baptiste Calmès, Emanuele Dotto, Yonatan Harpaz, Fabian Hebestreit, Markus
  Land, Kristian Moi, Denis Nardin, Thomas Nikolaus, and Wolfgang Steimle.
\newblock Hermitian {K}-theory for stable $\infty$-categories {III}:
  Grothendieck-witt groups of rings, 2021.
\newblock URL: \url{https://arxiv.org/abs/2009.07225}.

\bibitem{CC}
G.~E. Carlsson and R.~L. Cohen.
\newblock The cyclic groups and the free loop space.
\newblock {\em Comment. Math. Helv.}, 62:423--449, 1987.

\bibitem{CS}
M.~Chas and D.~Sullivan.
\newblock String topology, 1999.
\newblock URL: \url{https://arxiv.org/abs/math/9911159}.

\bibitem{CHV}
Ralph~L. Cohen, Kathryn Hess, and Alexander~A. Voronov.
\newblock {\em String topology and cyclic homology}.
\newblock Advanced Courses in Mathematics. CRM Barcelona. Birkh\"{a}user
  Verlag, Basel, 2006.
\newblock Lectures from the Summer School held in Almer\'{\i}a, September
  16--20, 2003.

\bibitem{Con}
Alain Connes.
\newblock Cohomologie cyclique et foncteurs {${\rm Ext}\sp n$}.
\newblock {\em C. R. Acad. Sci. Paris S\'er. I Math.}, 296(23):953--958, 1983.

\bibitem{Costello}
Kevin {Costello}.
\newblock {Topological conformal field theories and Calabi-Yau categories}.
\newblock {\em {Adv. Math.}}, 210(1):165--214, 2007.

\bibitem{DasSaha}
Apurba {Das} and Ripan {Saha}.
\newblock {Involutive and oriented dendriform algebras}.
\newblock {\em {J. Algebra}}, 581:63--91, 2021.

\bibitem{Dennis-Igusa}
R.~Keith Dennis and Kiyoshi Igusa.
\newblock Hochschild homology and the second obstruction for pseudoisotopy.
\newblock Algebraic {{\(K\)}}-theory, {Proc}. {Conf}., {Oberwolfach} 1980,
  {Part} {I}, {Lect}. {Notes} {Math}. 966, 7-58 (1982)., 1982.

\bibitem{dotto-thesis}
Emanuele Dotto.
\newblock Stable real {K}-theory and real topological hochschild homology,
  2012.
\newblock URL: \url{https://arxiv.org/abs/1212.4310}.

\bibitem{Dotto-equivariant}
Emanuele Dotto.
\newblock Equivariant calculus of functions and
  {{\(\mathbb{Z}/2\)}}-analyticity of real algebraic {{\(K\)}}-theory.
\newblock {\em J. Inst. Math. Jussieu}, 15(4):829--883, 2016.

\bibitem{DMP2}
Emanuele Dotto, Kristian Moi, and Irakli Patchkoria.
\newblock On the geometric fixed-points of real topological cyclic homology,
  2023.
\newblock URL: \url{https://arxiv.org/abs/2106.04891}.

\bibitem{DMPR}
Emanuele Dotto, Kristian Moi, Irakli Patchkoria, and Sune~Precht Reeh.
\newblock Real topological {H}ochschild homology.
\newblock {\em J. Eur. Math. Soc. (JEMS)}, 23(1):63--152, 2021.

\bibitem{DO}
Emanuele Dotto and Crichton Ogle.
\newblock {$K$}-theory of {H}ermitian {M}ackey functors, real traces, and
  assembly.
\newblock {\em Ann. K-Theory}, 4(2):243--316, 2019.

\bibitem{DMP}
Emanuele Dotto, Irakli Patchkoria, and Kristian Jonsson~Moi.
\newblock Witt vectors, polynomial maps, and real topological {H}ochschild
  homology.
\newblock {\em Ann. Sci. \'{E}c. Norm. Sup\'{e}r. (4)}, 55(2):473--535, 2022.

\bibitem{Dunn}
Gerald Dunn.
\newblock Dihedral and quaternionic homology and mapping spaces.
\newblock {\em $K$-Theory}, 3(2):141--161, 1989.

\bibitem{DK}
T.~Dyckerhoff and M.~Kapranov.
\newblock Crossed simplicial groups and structured surfaces.
\newblock In {\em Stacks and categories in geometry, topology, and algebra.
  CATS4 conference on higher categorical structures and their interactions with
  algebraic geometry, algebraic topology and algebra, CIRM, Luminy, France,
  July 2--7, 2012.}, pages 37--110. Providence, RI: American Mathematical
  Society (AMS), 2015.

\bibitem{Dyer-Eilenberg}
Eldon Dyer and S.~Eilenberg.
\newblock An adjunction theorem for locally equiconnected spaces.
\newblock {\em Pac. J. Math.}, 41:669--685, 1972.

\bibitem{FS}
Marco Farinati and Andrea Solotar.
\newblock Morita equivalence for positive {Hochschild} homology and dihedral
  homology.
\newblock {\em Commun. Algebra}, 24(5):1793--1807, 1996.

\bibitem{FV}
Rams\`es {Fern\`andez-Val\`encia}.
\newblock {On the structure of unoriented topological conformal field
  theories}.
\newblock {\em {Geom. Dedicata}}, 189:113--138, 2017.

\bibitem{RFG}
Rams\`es {Fern\`andez-Val\`encia} and Jeffrey {Giansiracusa}.
\newblock {On the Hochschild homology of involutive algebras}.
\newblock {\em {Glasg. Math. J.}}, 60(1):187--198, 2018.

\bibitem{Fie}
Z.~Fiedorowicz.
\newblock The symmetric bar construction.
\newblock URL: \url{https://people.math.osu.edu/fiedorowicz.1/}.

\bibitem{FL}
Zbigniew Fiedorowicz and Jean-Louis Loday.
\newblock Crossed simplicial groups and their associated homology.
\newblock {\em Trans. Amer. Math. Soc.}, 326(1):57--87, 1991.

\bibitem{Fox}
Ralph~H. Fox.
\newblock On fibre spaces. {I}, {II}.
\newblock {\em Bull. Am. Math. Soc.}, 49:555--557, 733--735, 1943.

\bibitem{F-Mc}
A.~Fr{\"o}hlich and A.~M. McEvett.
\newblock Forms over rings with involution.
\newblock {\em J. Algebra}, 12:79--104, 1969.

\bibitem{Lewis}
L.~Gaunce~{Lewis, Jr}.
\newblock When is the natural map {{\(X \to\Omega\Sigma X\)}} a cofibration?
\newblock {\em Trans. Am. Math. Soc.}, 273:147--155, 1982.

\bibitem{Goodwillie}
Thomas~G. {Goodwillie}.
\newblock {Cyclic homology, derivations, and the free loopspace}.
\newblock {\em {Topology}}, 24:187--215, 1985.

\bibitem{DMG-IFAS}
Daniel Graves.
\newblock {PROP}s for involutive monoids and involutive bimonoids.
\newblock {\em Theory Appl. Categ.}, 35:No. 42, 1564--1575, 2020.

\bibitem{DMG-equiv}
Daniel Graves.
\newblock Categorifying equivariant monoids, 2022.
\newblock URL: \url{https://arxiv.org/abs/2211.03076}.

\bibitem{DMG-PROB}
Daniel Graves.
\newblock Composing {PROB}s.
\newblock {\em Theory Appl. Categ.}, 38:No. 26, 1050--1061, 2022.

\bibitem{DMG-hyp}
Daniel Graves.
\newblock Hyperoctahedral homology for involutive algebras.
\newblock {\em Homology Homotopy Appl.}, 24(1):1--26, 2022.

\bibitem{DMG-e-inf}
Daniel Graves.
\newblock E-infinity structure in hyperoctahedral homology.
\newblock {\em Homology Homotopy Appl.}, 25(1):1--19, 2023.

\bibitem{GM}
Detlef Gromoll and Wolfgang Meyer.
\newblock Periodic geodesics on compact riemannian manifolds.
\newblock {\em J. Differential Geometry}, 3:493--510, 1969.

\bibitem{Hahn}
Alexander~J. Hahn.
\newblock A {Hermitian} {Morita} theorem for algebras with anti-structure.
\newblock {\em J. Algebra}, 93:215--235, 1985.

\bibitem{HM}
L.~Hesselholt and I.~Madsen.
\newblock Real algebraic $k$-theory, 2015.
\newblock URL: \url{https://web.math.ku.dk/~larsh/papers/s05/}.

\bibitem{Hogenhaven1}
A.~H{\o}genhaven.
\newblock Real topological cyclic homology of spherical group rings, 2016.
\newblock URL: \url{https://arxiv.org/abs/1611.01204}.

\bibitem{Hogenhaven2}
A.~H{\o}genhaven.
\newblock On the geometric fixed points of real topological hochschild
  homology, 2017.
\newblock URL: \url{https://arxiv.org/abs/1710.01817}.

\bibitem{Jones}
John D.~S. Jones.
\newblock Cyclic homology and equivariant homology.
\newblock {\em Invent. Math.}, 87:403--423, 1987.

\bibitem{Kras-skew}
R.~Krasauskas.
\newblock Skew-simplicial groups.
\newblock {\em Litovsk. Mat. Sb.}, 27(1):89--99, 1987.

\bibitem{KLS-dihed}
R.~L. Krasauskas, S.~V. Lapin, and Yu.~P. Solov'ev.
\newblock Dihedral homology and cohomology. {B}asic concepts and constructions.
\newblock {\em Mat. Sb. (N.S.)}, 133(175)(1):25--48, 143, 1987.

\bibitem{Lack}
Stephen Lack.
\newblock Composing {PROPS}.
\newblock {\em Theory Appl. Categ.}, 13:No. 9, 147--163, 2004.

\bibitem{Lod-dihed}
Jean-Louis Loday.
\newblock Homologies di\'{e}drale et quaternionique.
\newblock {\em Adv. in Math.}, 66(2):119--148, 1987.

\bibitem{Lod}
Jean-Louis Loday.
\newblock {\em Cyclic homology}, volume 301 of {\em Grundlehren der
  Mathematischen Wissenschaften}.
\newblock Springer-Verlag, Berlin, second edition, 1998.

\bibitem{LQ}
Jean-Louis {Loday} and Daniel {Quillen}.
\newblock {Cyclic homology and the Lie algebra homology of matrices}.
\newblock {\em {Comment. Math. Helv.}}, 59:565--591, 1984.

\bibitem{Lodder1}
Gerald~M. Lodder.
\newblock Dihedral homology and the free loop space.
\newblock {\em Proc. London Math. Soc. (3)}, 60(1):201--224, 1990.

\bibitem{LF}
L.~A. Lyusternik and A.~I. Fet.
\newblock Variational problems on closed manifolds.
\newblock {\em Doklady Akad. Nauk SSSR (N.S.)}, 81:17--18, 1951.

\bibitem{MCM}
S.~{Mac Lane}.
\newblock {The Milgram bar construction as a tensor product of functors}.
\newblock In {\em The Steenrod Algebra and Its Applications: A Conference to
  Celebrate N.E. Steenrod's Sixtieth Birthday}, volume 168 of {\em Lecture
  Notes in Math.}, pages 135--152. Springer, Berlin, 1970.

\bibitem{CWM}
Saunders Mac~Lane.
\newblock {\em Categories for the working mathematician}, volume~5 of {\em
  Graduate Texts in Mathematics}.
\newblock Springer-Verlag, New York, second edition, 1998.

\bibitem{MMSS}
M.~A. Mandell, J.~P. May, S.~Schwede, and B.~Shipley.
\newblock Model categories of diagram spectra.
\newblock {\em Proc. Lond. Math. Soc. (3)}, 82(2):441--512, 2001.

\bibitem{may-fibrations}
J.~Peter May.
\newblock Classifying spaces and fibrations.
\newblock {\em Mem. Amer. Math. Soc.}, 1(1, 155):xiii+98, 1975.

\bibitem{Milnor}
J.~Milnor.
\newblock On the construction {$FK$}.
\newblock In {\em J. F. Adams Algebraic topology. {A} student's guide},
  volume~4 of {\em Lond. Math. Soc. Lect. Note Ser.}, pages 119--136. Cambridge
  University Press, Cambridge. London Mathematical Society, London, 1972.

\bibitem{NS}
Thomas Nikolaus and Peter Scholze.
\newblock On topological cyclic homology.
\newblock {\em Acta Math.}, 221(2):203--409, 2018.

\bibitem{Oancea}
Alexandru Oancea.
\newblock Morse theory, closed geodesics, and the homology of free loop spaces.
\newblock In {\em Free loop spaces in geometry and topology}, volume~24 of {\em
  IRMA Lect. Math. Theor. Phys.}, pages 67--109. Eur. Math. Soc., Z\"{u}rich,
  2015.
\newblock With an appendix by Umberto Hryniewicz.

\bibitem{Pir02}
T.~Pirashvili and B.~Richter.
\newblock Hochschild and cyclic homology via functor homology.
\newblock {\em $K$-Theory}, 25(1):39--49, 2002.

\bibitem{Pir-PROP}
Teimuraz Pirashvili.
\newblock On the {PROP} corresponding to bialgebras.
\newblock {\em Cah. Topol. G\'{e}om. Diff\'{e}r. Cat\'{e}g.}, 43(3):221--239,
  2002.

\bibitem{Serre}
Jean-Pierre Serre.
\newblock Homologie singuli\`ere des espaces fibr\'{e}s. {A}pplications.
\newblock {\em Ann. of Math. (2)}, 54:425--505, 1951.

\bibitem{spalinski2}
Jan Spali\'{n}ski.
\newblock Homotopy theory of dihedral and quaternionic sets.
\newblock {\em Topology}, 39(3):557--572, 2000.

\bibitem{spal-discrete}
Jan {Spali\'nski}.
\newblock {A discrete model of \(O(2)\)-homotopy theory}.
\newblock {\em {J. Pure Appl. Algebra}}, 214(1):1--5, 2010.

\bibitem{Strom}
Arne Str{\o}m.
\newblock Note on cofibrations.
\newblock {\em Math. Scand.}, 19:11--14, 1966.

\bibitem{Ung}
M.~Ungheretti.
\newblock Free loop spaces and dihedral homology, 2016.
\newblock URL: \url{https://arxiv.org/abs/1608.08140}.

\bibitem{VP-B}
Micheline Vigu{\'e}-Poirrier and Dan Burghelea.
\newblock A model for cyclic homology and algebraic {K}-theory of 1-connected
  topological spaces.
\newblock {\em J. Differ. Geom.}, 22:243--253, 1985.

\bibitem{weib}
Charles~A. Weibel.
\newblock {\em An introduction to homological algebra}, volume~38 of {\em
  Cambridge Studies in Advanced Mathematics}.
\newblock Cambridge University Press, Cambridge, 1994.

\end{thebibliography}
\end{document}